\documentclass{scrartcl}
\usepackage[backend=bibtex,firstinits=true,doi=false,isbn=false,url=false]{biblatex}
\bibliography{literature}
\DeclareNameAlias{default}{last-first}
\usepackage{amssymb,amsmath,amsthm}
\usepackage{dsfont}
\usepackage[shortlabels]{enumitem}
\setlist{labelindent=\parindent,leftmargin=*}
\usepackage{authblk}
\numberwithin{equation}{section}
\theoremstyle{definition}
\newtheorem{rem}{Remark}[section]
\theoremstyle{definition}
\newtheorem{defi}[rem]{Definition}
\theoremstyle{plain}
\newtheorem{theo}[rem]{Theorem}
\newtheorem{prop}[rem]{Proposition}
\newtheorem{lemma}[rem]{Lemma}
\newtheorem{cor}[rem]{Corollary}
\newcommand\1{\mathds 1}

\newcommand\C{\mathbb C}

\newcommand\N{\mathbb N}

\newcommand\R{\mathbb R}

\newcommand\shd{{\mathcal D}}
\newcommand\shf{{\mathcal F}}
\newcommand\shg{{\mathcal G}}

\newcommand\shm{{\mathcal M}}
\newcommand\shn{{\mathcal N}}

\newcommand\shs{{\mathcal S}}

\newcommand{\rz}[1]{\mathbb{#1}}

\DeclareMathOperator\MP{MP}

\title{Elliptic PDEs with distributional drift and  backward SDEs
driven by a c\`adl\`ag martingale with random terminal time}
\author{Francesco Russo}
\affil{Ecole Nationale Sup\'erieure des Techniques Avanc\'ees\\ ENSTA ParisTech\\Unit\'e de Math\'ematiques appliqu\'ees}
\author{Lukas Wurzer}
\affil{University of Innsbruck\\ Institute for Basic Sciences in Engineering Sciences\\Unit for Engineering Mathematics}
\date{June 2nd, 2015}
\begin{document}
\maketitle

\begin{abstract}
  We introduce a generalized notion of semilinear elliptic partial
  differential equations where the corresponding second order partial
  differential operator $L$ has a generalized drift. We investigate existence
  and uniqueness of generalized solutions of class $C^1$.  The generator $L$
  is associated with a Markov process $X$ which is the solution of a
  stochastic differential equation with distributional drift. If the
  semilinear PDE admits boundary conditions, its solution is naturally
  associated with a backward stochastic differential equation (BSDE) with
  random terminal time, where the forward process is $X$.  Since $X$ is a weak
  solution of the forward SDE, the BSDE appears naturally to be driven by a
  martingale.  In the paper we also discuss the uniqueness of solutions of a
  BSDE with random terminal time when the driving process is a general
  c\`adl\`ag martingale.
\end{abstract}



\section{Introduction}

The paper involves three essential areas of study.
\begin{enumerate}
\item Elliptic semilinear PDEs with distributional drift.
\item Backward stochastic differential equations (BSDEs) driven by c\`adl\`ag
  martingales with terminal condition at random terminal time.
\item The representation of solutions of the above mentioned BSDEs through
  solutions of PDEs.
\end{enumerate}
We consider a
differential equation of the type
\begin{equation}\label {EEllip}
  Lu=F\left(x,u,u'\right),
\end{equation}
on $[0,1]$ with boundary conditions, where $L$ is the generator of a
one-dimensional stochastic differential equation of the type $
Lg=\frac{\sigma^2}{2}g''+\beta'g'$, with $\sigma, \beta$ being real continuous
functions and $\sigma$ is strictly positive. In general, \eqref{EEllip} is a
semilinear PDE, which reduces to an ODE in the case of one dimension
considered here. The drift $\beta'$ is the derivative of a continuous function
$\beta$, in general a distribution.  A typical example of such $\beta$ is the
path of a fixed continuous process.  $F$ is a continuous real function defined
on $\mathbb [0,1] \times \R^2$. When $F$ does not depend on $u$ and $u'$, and
$x$ varies on the real line, \eqref{EEllip} was introduced in
\cite{frw1,frw2}, via the notion of $C^1$-solutions which appear as limit of
solutions of elliptic problems with regularized coefficients.  Indeed
\cite{frw1,frw2} investigated the case of initial conditions.

One-dimensional stochastic differential equations with distributional drift
were examined by several authors, see \cite{frw1, frw2, BQ, trutnau} and
references therein, with a recent contribution by \cite{karatzas}.  Such an
equation appears formally as
\begin{equation} \label{SDE} dX_t = \beta'(X_t) dt + \sigma(X_t) dW_t.
\end{equation}
More recently some contributions also appeared in the multidimensional case,
see \cite{basschen2}, when the drift is a Kato class measure and in
\cite{issoglio} for other type of time dependent drifts.

A motivation for studying the mentioned type of equations comes from the
literature of \emph{random media}.  A special case of equation \eqref{SDE} with
$a = 1$ and $\beta$ being the continuous function $\beta$ was considered by
several authors, see e.g.  \cite{mat1, mat2, hushi, seignourel}, in particular
in relation with long time behavior, without defining the stochastic analysis
framework.  In that case, the solution $X$ of \eqref{SDE} is the so called \emph{Brox diffusion}.  The discrete version of this is the random walk in random
environment.  In that case the solution $X$ describes the motion of a particle
in an \emph{irregular medium}: the velocity of the medium $\beta'$ can be for
instance the realization of a Gaussian white noise but the noise could be
 also of
other nature.  In particular $\beta$ is often a (possibly fractional) Brownian
path, but it could be the path any continuous process.


This paper is devoted to the following main objectives.
\begin{enumerate}
\item We study existence and uniqueness of a solution $u$ of the semilinear
  equation \eqref{EEllip} with prescribed initial conditions for $u(0)$ and
  $u'(0)$, see Proposition~\ref{P34}.
\item We show that the initial value problem allows to provide a solution to
  the boundary value problem on $[0,1]$ for \eqref{EEllip}, see
  Proposition~\ref{pdex}.
\item We explore several assumptions on $F$ which provide existence and/or
  uniqueness of solutions to the boundary value problem, see
  Corollary~\ref{CEE} and Propositions~\ref{pdex} and~\ref{FP}.
\item We study the uniqueness of solutions of BSDEs driven by a c\`adl\`ag
  martingale $M$ such that $\left< M \right>$ is continuous, see
  Theorem~\ref{thu}.
\item We show that a solution of the PDE \eqref{EEllip} with Dirichlet
  boundary conditions on $[0,1]$ generates a solution to a special forward
  BSDE (see Theorem~\ref{PB4}) with terminal condition at the random time
  $\tau$, where $\tau$ is the exit time from $[0,1]$ of a solution $X$ of an
  SDE with distributional drift.
\item Those solutions which are associated with \eqref{EEllip} are the unique
  solutions of the corresponding BSDE (in some reasonable class) whenever $F$
  fulfills in particular some strict monotonicity condition in the second
  variable, i.\,e.\ \eqref{mono} holds.
\item We illustrate situations where the BSDE admits no uniqueness in a
  reasonable class but the probabilistic representation still holds.
\end{enumerate}
As we mentioned, a significant object of study is a backward SDE with random
terminal time, which was studied and introduced by \cite{dapa} when the
driving martingale is a Brownian motion. BSDEs driven by a c\`adl\`ag
martingale with fixed time terminal time were studied in \cite{buckdahn,
  elkarouihuang, sant, Briand2002robustness}.

The paper is organized as follows. After the introduction in Section~\ref{S2},
we recall some preliminaries about linear elliptic differential equations with
initial condition and the notion of martingale problem related to an SDE with
distributional drift.  In Section~\ref{S3}, we discuss existence and
uniqueness of solutions to \eqref{EEllip}, in Section~\ref{S4} we discuss the
first exit time properties of a solution to equation \eqref{SDE}.  In
Section~\ref{S5} we investigate uniqueness for BSDEs with random terminal
condition with related probabilistic representation.  Finally
Section~\ref{sec:existence} shows how a solution to \eqref{EEllip} generates a
solution to a special BSDE with terminal condition at random time.

\section{Preliminaries}
\label{S2}


\subsection{The linear elliptic PDE with distributional drift}

\label{S22}

If $I$ is a real open interval, then $C^0(I)$ will be the space
of continuous functions on $I$ endowed with the topology of uniform
convergence on compacts. For $k\ge 1$, $C^k(I)$ will be the space of $k$-times
continuously differentiable functions on $I$, equipped with the topology of
uniform convergence of the first $k$ derivatives. If $I=\R, k \ge 0$, 
then we will
simply write  $C^k$ instead of  $C^k(\R)$.
If $I = [a,b]$ with $ -\infty < a < b < +\infty$, then $u:I \to \R$ is said to
be of class $C^1([a,b])$ if it is of class $C^1(\left]a,b\right[)$ and if the derivative
extends continuously to $[a,b]$.

In this section we introduce the ``generator'' $L$ of our diffusion with
distributional drift adopting the notations and conventions of
\cite{frw1,frw2}.

Let $\sigma$, $ \beta \in C^0$ such that $\sigma>0$. We consider formally a
differential operator of the following type \cite[section 2]{frw1}:
\begin{equation}\label{E2.1}
  Lg=\frac{\sigma^2}{2}g''+\beta' g'.
\end{equation}
By a mollifier, we intend a function $\Phi$ belonging to the Schwartz space
$\shs\left(\R\right)$ with $\int\Phi\left(x\right)\,dx=1$. We denote
\[\Phi_n\left(x\right):=n
\Phi(nx),\quad\sigma_n^2:=\sigma^2\ast\Phi_n,\quad \beta_n:=\beta
\ast\Phi_n.\] We then consider
\begin{equation}\label{E2.2}
  L_n g=\frac{\sigma_n^2}{2}g''+\beta_n'g'.
\end{equation}
A priori, $\sigma_n^2$, $\beta_n$ and the operator $L_n$ depend on the
mollifier $\Phi$.
\begin{defi}
  \label{D31} Let $l\in C^0$.  A function $f\in C^1(\R)$ is said to be a
  $C^1$-\emph{solution} to
  \begin{equation}\label{E2.3}
    Lf=l,
  \end{equation}
  if, for any mollifier $\Phi$, there are sequences $\left(f_n\right)$ in
  $C^2$, $\left(l_n\right)$ in $C^0$ such that
  \begin{equation}\label{E2.4}
    L_nf_n=l_n,\quad f_n\to f\ \text{in}\ C^1,\quad l_n\to l\ \text{in}\ C^0.
  \end{equation}
\end{defi}
The following proposition gives conditions for the existence of a solution $h$
to the homogeneous version of \eqref{E2.3}, see \cite[Prop. 2.3]{frw1}.
\begin{prop}\label{T24}
  Let $a \in \R $ be fixed.  There is a $C^1$-solution to $Lh=0$ such that
  $h'\left(x\right)\neq 0$ for every $x\in\R$ if and only if
  \begin{equation} \label{ESigma}
    \Sigma\left(x\right):=\lim_{n\to\infty}2\int_a^x\frac{\beta_n'}{\sigma_n^2}(y)\,dy
  \end{equation} 
  exists in $C^0$, independently from the mollifier. Moreover, in this case,
  any $C^1$-solution $f$ to $Lf=0$ fulfills
  \begin{equation}\label{E2.5}
    f'\left(x\right)=e^{-\Sigma\left(x\right)}f'(a), \ \forall x \in  \R.
  \end{equation}
\end{prop}
\begin{rem} \label{R2.5}
\ 
  \begin{enumerate}
  \item In particular, this proves the uniqueness of solutions to the problem
    \begin{equation}\label{E2.9}
      Lf=l,\quad f\in C^1,\quad f\left(a\right)=x_0,\quad 
      f'\left(a\right)=x_1, 
    \end{equation}
    for every $l\in C^0$, $x_0,x_1\in\R$.
  \item In most of the cases we will set $a =0$.
  \end{enumerate}
\end{rem}
In the sequel we will always suppose the existence of $\Sigma$ as in
\eqref{ESigma}.  We will denote $h:\R\to\R$ such that $h(0) = 0$ and
$h'=\exp\left(-\Sigma\right)$ and $h_n:\R\to\R$ so that
$h_n=\exp\left(-\Sigma_n\right)$ with
$\Sigma_n=2\int_0^x\frac{\beta'_n}{\sigma_n^2}\left(y\right)dy$.
The proposition below is a consequence of    \cite[Lemma 2.6]{frw1} and
    \cite[Remark 2.7]{frw1}.
\begin{prop} \label{R2.8} Let $a \in \R$ and $l\in C^0$ and $x_0, x_1\in\R$.
  Then there is a unique $C^1$-solution
  to
  \begin{align}\label{E2.10}
    Lu&=l\\
    \nonumber u(a)&=x_0,\quad u'(a)=x_1.
  \end{align}
  The solution satisfies
  \begin{equation}
    \label{E2.11}
    u'\left(x\right) = e^{-\Sigma\left(x\right)} \left( 2 \int_a^xe^{\Sigma(y)}
      \frac{l\left(y\right)}{\sigma^2\left(y\right)}\, dy + x_1
    \right).
  \end{equation}
\end{prop}

We will denote by $\shd_L$
the set of all $f\in C^1$ which are $C^1$-solutions of
$Lf = l$ for some $l\in C^0$.
This defines without ambiguity $L:\shd_L \to C^0$.

\subsection{Related martingale problem}
\label{SFSDE}

For the moment we fix a probability space $(\Omega, {\shg},\rz{P})$.  All
processes will be considered on the index set $\rz{R}_+$.

For convenience, we follow the framework of stochastic calculus introduced in
\cite{rv1} and developed in several papers. A survey of that calculus in
finite dimension is given in~\cite{RVSem}.
We will fix a filtration $\shf = (\shf_t)$ which will fulfill the usual
conditions.

The covariation of two continuous processes $X$ and $Y$ is defined as follows.
Suppose that
\begin{equation}\label{E1.1}
  A_t:=\lim_{\varepsilon\to 0+}C_{\varepsilon}(X,Y)_t
\end{equation}
exists for any $t \in [0,T]$ in probability, where
\begin{equation*}
  C_\varepsilon(X,Y)_t:=\frac{1}{\varepsilon}\int_0^t\left(X_{s+\varepsilon}-X_s\right)\left(Y_{s+\varepsilon}-Y_s\right)\,ds.
\end{equation*}
We say that $(X,Y)$ admit a covariation if the random function $(A_t)$ admits
a (necessarily unique) continuous version, which will be designated by
$[X,Y]$.  For $\left[X,X\right]$ we often shortly write $[X]$.  All
the covariation processes will be continuous.
\begin{rem}\label{R1.2}
  In \cite[Propositions 1, 9 and 11, Remarks 1 and 2]{RVSem} we can find the
  following.
  \begin{enumerate}[a)]
  \item If $\left[X,X\right]$ exists, then it is always an increasing process
    and $X$ is called a \emph{finite quadratic variation process}. If
    $\left[X,X\right]\equiv 0$, then $X$ is said to be a \emph{zero quadratic
      variation process}.
  \item \label{R1.2b}Let $X$ and $Y$ be continuous processes such that
    $[X,Y]$, $\left[X,X\right]$, $\left[Y,Y\right]$ exist. Then $[X,Y]$ is a
    bounded variation process. If $f,g\in C^1$, then
    \begin{equation*}
      \left[f(X),g(Y)\right]_t=\int_0^t
      f'(X)g'(Y)\,d[X,Y].
    \end{equation*}
  \item \label{R1.2c} If $A$ is a zero quadratic variation process and $X$ is
    a finite quadratic variation process, then $\left[X,A\right]\equiv 0$.
  \item \label{R1.2d} A bounded variation process is a zero quadratic
    variation process.
  \item \label{R1.2e}
    If $M$ and $N$ are $\shf$-local martingales, then $[M,N]$ is the usual
    covariation process 
$\left<M,N\right>$.
  \end{enumerate}
\end{rem}

An $\shf$-\emph{Dirichlet process} is the sum of an $\shf$-local continuous
martingale $M$ and an $\shf$-adapted zero quadratic variation process $A$, see
\cite{fo,ber}.
\begin{rem}\label{R1.3}
  Let $X=M+A$ be an $\shf$-Dirichlet process.
  \begin{enumerate}
  \item Remark~\ref{R1.2}\ref{R1.2c} and~\ref{R1.2e} together with the
    bilinearity of the covariation operator imply that $[X]=\left<M\right>$.
  \item If $f\in C^1$, then $f(X)=M^f+A^f$ is an $\shf$-Dirichlet process,
    where
    \begin{equation*}
      M^f=\int_0^\cdot f'(X_s)\, dM_s
    \end{equation*} and $A^f:=f(X)-M^f$ has zero quadratic
    variation. This easily follows from the bilinearity of covariation and
    Remark~\ref{R1.2}\ref{R1.2b},~\ref{R1.2c} and~\ref{R1.2e}. 
    See also \cite{ber} for a similar result and
    Proposition 17 in \cite{RVSem} for a generalization
    to weak Dirichlet processes.
  \end{enumerate}
\end{rem}
\begin{defi}
  Given a stopping time $\tau$ and a process $X$, we denote by $X^\tau$ the
  \emph{stopped process}
  \begin{equation*}
    X_t^\tau:=X_{t\wedge\tau},\qquad t\geq 0.
  \end{equation*}
\end{defi}
\begin{rem}\label{R1.4} Let $\tau$ be an $\shf$-stopping
  time.
  If $X$ is an $\shf$-semimartingale (resp.\ $\shf$-Dirichlet process), then
  the stopped processes $X^{\tau}$ is also a semimartingale (resp.\
  $\shf$-Dirichlet process).
\end{rem}
In the classical theory of Stroock and Varadhan, see e.\,g.\ \cite{stroock},
the solutions of martingale problems are probabilities on the canonical space
$C^0([0,T])$ equipped with its Borel $\sigma$-field and the Wiener measure. Here
the meaning is a bit different since the solutions are considered to be
processes. For the sequel of the section we fix  $x_0 \in \R$.

\begin{defi} \label{DMartPr}
A process $X$ (defined on some probability
  space), is said to solve the \emph{martingale problem} $\MP\left(\sigma,
    \beta; x_0\right)$ 
if $X_0 = x_0$ a.s. and
  \begin{equation} \label{E12}
    f(X_t)-f\left(x_0\right)-\int_0^tLf(X_s)ds
  \end{equation}
  is a local martingale for any $f\in\shd_L$.

\end{defi}

In the sequel we will denote by $\shf^X = (\shf^X_t)$ the canonical filtration
associated with $X$.

\begin{defi} \label{DULaw} We say that the martingale problem $\MP(\sigma,
  \beta;x_0)$ admits \emph{uniqueness (in law)} if any processes $X_1$ and
  $X_2$, defined on some probability space and solving the martingale problem,
  have the same law.
\end{defi}
The proposition below was the object of Proposition 3.13 of \cite{frw1}.
\begin{prop}\label{P3.13}
  Let $v$ be the unique solution to $Lv=1$ in the $C^1$-sense such that $
  v\left(0\right)=v'\left(0\right)=0$. Then
  there exists a unique (in law) solution 
 the martingale problem $\MP(\sigma,  \beta;x_0)$
if and only if
  \begin{equation}
    \label{4.2}
    v\left(-\infty\right)=v\left(+\infty\right)=+\infty.
  \end{equation}
\end{prop}
In several contexts (see \cite{frw1}) the solution of the previous martingale
problem turns out to be a solution (in the proper sense) of \eqref{SDE}, but
it will not be used in this paper.

Proposition~\ref{P3.13} implies the following.
\begin{prop} \label{4.2bis}
The martingale problem $\MP\left(\sigma,\beta;x_0\right)$
  admits exactly one solution in law if and only if the function $v:\R\to\R$
  defined by
  \begin{equation}\label{F1}
    \begin{aligned}
      v(0)&=0,\\
      v'\left(x\right)&=e^{-\Sigma\left(x\right)}\left(2\int_0^x\frac{1}{\sigma^2}\left(y\right)dy\right)
    \end{aligned}
  \end{equation}
  fulfills
\begin{equation} \label{F2}
 v\left(-\infty\right)=v\left(+\infty\right)=+\infty.
\end{equation}
\end{prop}
\begin{proof}
  This follows from Proposition~\ref{P3.13}, Proposition~\ref{R2.8} and from
  the fact that $v$ defined in \eqref{F1} is the solution of the problem
  \begin{align*}
    Lv&=1,&v\left(0\right)&=v'\left(0\right)=0.
  \end{align*}
\end{proof}
From now on Assumption \eqref{F2} for the function $v$ defined by \eqref{F1}
will always be in force. Let then $X$ be a solution to the martingale problem
on a suitable probability space and $\shf^X$ be its canonical filtration.
\begin{rem} \label{RR} \
  \begin{enumerate}[i)]
  \item \label{RR1} By Remark 3.3 of \cite{frw1}, choosing $f$ as the identity
    function, $X$ is an $\shf^X$-Dirichlet process, whose local martingale
    part $M^X$ verifies
    \[\left[M^X\right]_t=\int_0^t\sigma^2(X_s)ds.\]
  \item Consequently by
 Remark~\ref{R1.2}\ref{R1.2c}  and \ref{R1.2e}
together with  the bilinearity of
    covariation it follows $[X]_t=\int_0^t\sigma^2(X_s)ds$.
  \end{enumerate}
\end{rem}
\begin{prop}
  \label{P48}
  Let $X$ be a solution of $\MP(\sigma, \beta;x_0)$.
  For every $\varphi\in\shd_L$ we have
  \begin{equation*}
    \varphi(X_t)=\varphi(X_0)+\int_0^t\varphi'(X_s)dM_s^X+\int_0^t(L\varphi)(X_s)ds.
  \end{equation*}
\end{prop}
\begin{proof}
  By definition of the martingale problem
  there is an $\shf^X$-local martingale $M^\varphi$ such that
  \begin{equation}
    \label{E481}
    \varphi(X_t)=\varphi(X_0)+M^\varphi_t+\int_0^t(L\varphi)(X_s)ds.
  \end{equation}
  On the other hand, by Remark~\ref{RR}\ref{RR1}
  and Remark~\ref{R1.3} $\varphi(X_t)$ is an $\shf^X$-Dirichlet process with
  decomposition
  \begin{equation}
    \label{E482}
    \varphi(X_t)=\varphi(X_0)+\int_0^t\varphi'(X_s)dM^X_s+A_t^\varphi,
  \end{equation}
  where $\left[A^\varphi\right]\equiv 0$. By the uniqueness of Dirichlet
  decomposition and the identification of \eqref{E481} and \eqref{E482} the
  result follows.
\end{proof}

\section{The semilinear elliptic PDE with distributional drift and boundary
  conditions}
\label{S3}

In this section we present the deterministic analytical framework that we will
need in the paper.

\subsection{The linear case}

We explain here how to transform the study of our initial value problem
to a boundary value problem.
\begin{defi}
  \label{DLP1}
  Let $a,b,A,B\in\R$, such that $-\infty<a<b<\infty$.
  Additionally, let $g:[a,b]\to\R$ be continuous. We say that $u:[a,b]\to\R$
  is a \emph{solution of the boundary value problem}
  \begin{equation}
    \label{EP1}
    \left\{
      \begin{aligned}
        Lu&=g,\\
        u(a)&=A,\\
        u(b)&=B,
      \end{aligned}
    \right.
  \end{equation}
  if there is a continuous extension $\tilde g:\R\to\R$ of $g$ and a function
  $\tilde u\in\shd_L$ fulfilling $\tilde u\left|_{[a,b]}\right.=u$, such that
  $\tilde u$ is a solution of
  \begin{equation}
    \label{EP2}
    L\tilde u=\tilde g,
  \end{equation}
  in the sense of Definition~\ref{D31}, and $\tilde u(a)=A$, $\tilde u(b)=B$.
\end{defi}
\begin{prop}
  \label{P11}
  Let $g:[0,1]\to\R$ be continuous, $A,B\in\R$, $a=0$ and $b=1$. Then there
  exists a unique solution $u$ to \eqref{EP1}, given by
  \begin{subequations}
    \label{exp}
    \begin{align}
      u(x)&=f(x)+\int_0^1K(x,y)g(y)dy\\
      f(x)&:=\frac{B\int_0^xdye^{-\Sigma(y)}+A\int_x^1dye^{-\Sigma(y)}}{\int_0^1dye^{-\Sigma(y)}}\label{exp2}\\
      K(x,y)&:=\1_{y\leq
        x}\frac{2e^{\Sigma(y)}}{\sigma^2(y)}\int_y^xdze^{-\Sigma(z)}-2\frac{\int_0^xdre^{-\Sigma(r)}}{\int_0^1dre^{-\Sigma(r)}}\frac{e^{\Sigma(y)}}{\sigma^2(y)}\int_y^1dze^{-\Sigma(z)}.\label{exp3}
    \end{align}
  \end{subequations}
\end{prop}
\begin{rem} \label{Rexp} For every $y \in [0,1], x \mapsto K(x,y)$ is
  absolutely continuous and $(x,y) \mapsto \partial_x K(x,y)$ belongs to
  $L^\infty([0,1]^2)$.
\end{rem}
\begin{proof}[Proof of Proposition~\ref{P11}]
  We start with existence. Let $\tilde g$ be a continuous extension of $g$ and
  $x_1\in\R$. Then, by Proposition~\ref{R2.8}, there exists a unique solution
  $\tilde{u}$ to the problem on the real line,
  \begin{subequations}
    \label{ue}
    \begin{align}
      L\tilde u(x)&=\tilde g(x),\ x\in\R,\\
      \tilde u(0)&=A,\\
      \tilde u'(0)&=x_1,
    \end{align}
  \end{subequations}
  given by
  \begin{equation}
    \label{given}
    \tilde u(x)=A+\int_0^xe^{-\Sigma(y)}\left(2\int_0^ye^{\Sigma(z)}\frac{\tilde g(z)}{\sigma^2(z)}dz+x_1\right)dy.
  \end{equation}
  We look for $x_1\in\R$, so that $\tilde u(1)=B$. This gives
    \begin{align*}
      B&=A+x_1\int_0^1e^{-\Sigma(y)}dy+2\int_0^1dye^{-\Sigma(y)}\int_0^ydze^{\Sigma(z)}\frac{\tilde g(z)}{\sigma^2(z)},\\
      x_1&=\frac{B-A-2\int_0^1dze^{\Sigma(z)}\frac{\tilde
          g(z)}{\sigma^2(z)}\int_z^1dye^{-\Sigma(y)}}{\int_0^1e^{-\Sigma(y)}dy}.
    \end{align*}
  We insert $x_1$ into \eqref{given} and use the fact that $u=\tilde
  u|_{[0,1]}$ and $g=\tilde g|_{[0,1]}$. This gives \eqref{exp}, and we get
  $u(0)=A$ and $u(1)=B$.

  To show uniqueness, let $v^1$ and $v^2$ be two solutions of \eqref{EP1}, and
  set $v=v^1-v^2$. Then there is $\tilde v\in\shd_L$ with $\tilde
  v|_{[0,1]}=v$ and an $\tilde l\in C^0$ with $\tilde l|_{[0,1]}=0$, so that
    \begin{eqnarray*}
      L\tilde v(x)&=&\tilde l,
\\
      \tilde v(0)&=&\tilde v(1)=0.
    \end{eqnarray*}
  We need to show that $v\equiv 0$. By Proposition~\ref{R2.8} we get
  \begin{align*}
    \tilde v'(x)&=e^{-\Sigma(x)}\left(2\int_0^x\frac{\tilde
        l(y)}{\sigma^2(y)}e^{\Sigma(y)}dy+\tilde v'(0)\right)&\forall x&\in\R.
  \end{align*}
  In particular, since $\tilde l\left|_{[0,1]}\right.=0$, we get
  \begin{align*}
    \tilde v'(x)&=e^{-\Sigma(x)}\tilde v'(0),&\forall x&\in[0,1].
  \end{align*}
  Consequently, for $x\in[0,1]$,
  \begin{equation*}
    \tilde v(x)=\left(\int_0^xdy\,e^{-\Sigma(y)}\right)\tilde v'(0).
  \end{equation*}
  Since $\tilde v(1)=0$, it follows $\tilde v'(0)=0$ and so $v(x)=\tilde
  v(x)=0$ $\forall x\in[0,1]$.
\end{proof}

\subsection{Solution of the semilinear problem on the real line}
\label{SNLC}

We extend here the notion of $C^1$-solution to the semilinear case.
\begin{defi}
  \label{strongsolnonl}
  Let $F: \R \times \R^2\to\R$ be a continuous function.  We say that $u\in
  C^1$ is a $C^1$-\emph{solution} (on the real line) of
  \begin{equation}
    \label{pdenonlinear}
    Lu=F\left(x,u,u'\right)
  \end{equation}
  if $u$ is a $C^1$-solution of $ Lu= h$, with $h: \R \rightarrow \R$ defined
  by $h(x) = F\left(x,u(x),u'(x)\right)$.
\end{defi}
\begin{defi} \label{Lipschitz} 
Let  us consider a function $F: I \times \R \times \R \to \R $,
\begin{enumerate}
\item  $(x,y,z) \mapsto F(x,y,z)$ will be called \emph{globally Lipschitz} with
  respect to $z$ (resp. $(y,z)$)
  if $F$ is Lipschitz with respect to $z$ (resp. $(y,z)$) uniformly on $x$
  varying in $I$ and $y$ in $\R$ (resp.\ uniformly on $x$ varying in $I$).
  More precisely, $F$ is globally Lipschitz with respect to $z$ if there
  exists some constant $k$, called the \emph{Lipschitz constant for $F$}, such
  that
  \begin{equation}
    \label{lip}
    \left|F(x,y,z)-F\left(x,y,\tilde z\right)\right|\leq k \left|z-\tilde
      z\right|,\ \forall x\in I,\ \forall y,z,\tilde z\in\R.
  \end{equation}
  Similarly we speak about the Lipschitz constant $k$ related to a function
  $F$ which is globally Lipschitz with respect
  to $(y,z)$. 
\item Analogously $F$ will be said to have \emph{linear growth}
 with respect to $z$ (resp. $(y,z)$)
  if $F$ has linear growth with respect to $z$ (resp. $(y,z)$) uniformly on $x$
  varying in $I$ and $y$ in $\R$ (resp.\ uniformly on $x$ varying in $I$).
Obvious variants will also be used without further comment.  

\end{enumerate}

\end{defi}

\begin{prop} \label{P34}
  Suppose that $F: \R^3 \rightarrow \R$, so that $(x,y,z)\mapsto F(x,y,z)$
  restricted to $K \times \R^2$, for any compact interval $K$, is Lipschitz
  with respect to $(y,z)$.  Then there is a unique solution of
  \begin{equation}\label{PDEinit}
    \begin{aligned}
      Lu&=F\left(x,u(x),u'(x)\right),\qquad x\in\R,\\
      u(0)&=x_0,\\
      u'(0)&=x_1.
    \end{aligned}
  \end{equation}
\end{prop}
\begin{proof}
  By Proposition~\ref{R2.8}, $u:\R\to\R$ of class $C^1$ is a $C^1$-solution if
  and only if
  \begin{equation}\label{eq:strongsol}
    \begin{aligned}
      u'\left(x\right)&=e^{-\Sigma\left(x\right)}\left(2\int_0^x\frac{e^{\Sigma\left(y\right)}}{\sigma^2\left(y\right)}F\left(y,u\left(y\right),u'\left(y\right)\right)dy+x_1\right),\quad\forall x\in\R,\\
      u(0)&=x_0.
    \end{aligned}
  \end{equation}
  We can reduce the well-posedness of \eqref{eq:strongsol} to the
  well-posedness of
  \begin{equation}\label{eq:pn}
    \begin{aligned}
      u'\left(x\right)&=e^{-\Sigma\left(x\right)}\left(2\int_0^x\frac{e^{\Sigma\left(y\right)}}{\sigma^2\left(y\right)}F\left(y,u\left(y\right),u'\left(y\right)\right)dy+x_1\right),\forall x\in[-N,N],\\
      u(0)&=x_0,
    \end{aligned}
  \end{equation}
  for every $N\in\N^*:= \N - \{0\}$.  In the sequel of the proof, since \eqref{eq:pn}
  depends on $N$, we will often denote it by \eqref{eq:pn}$(N)$.

  Indeed, if $u_N$ is a solution of \eqref{eq:pn}$(N)$, then any solution of
  \eqref{eq:pn}$(N+1)$, restricted to $[-N,N]$ is a solution
  \eqref{eq:pn}$(N)$.  In this way the existence of a solution of
  \eqref{eq:strongsol} is equivalent to the existence of a family
  $\left(u_N\right)$ of functions which are respectively solutions of
  \eqref{eq:pn}$(N)$.  In the sequel we fix $N\in \mathbb N^*$ and we study
  existence and uniqueness for \eqref{eq:pn}$(N)$, which is an ODE in a
  compact interval. We consider the map $T:C^1\left([-N,N]\right)\to
  C^1\left(\left[-N,N\right]\right)$ defined by
    \begin{align*}
      Tf(0)&=x_0\\
      \left(Tf\right)'(x)&=e^{-\Sigma(x)}\left(2\int_0^x\frac{e^{\Sigma(y)}}{\sigma^2(y)}F\left(y,f(y),f'(y)
        \right) dy + x_1 \right).
    \end{align*}
  Clearly a function $u\in C^1\left([-N,N]\right)$ is a solution of
  \eqref{eq:pn}$(N)$ if and only if $Tu=u$.  $C^1\left([-N,N]\right)$ is a
  Banach space equipped with the norm
  \begin{equation*}
    \left\Vert f\right\Vert_{N}=\sup_{|x|\leq N}
    \left[| f(x)| + | f'(x)| \right].
  \end{equation*}
  The norm $ \left\Vert \cdot \right\Vert_{N}$ is equivalent to
  \begin{equation*}
    \left\Vert f\right\Vert_{N,\lambda}=\sup_{|x|\leq N}
    \left(| f(x) | +| f'(x)|
    \right)e^{\Sigma(x)-\lambda|x|},
  \end{equation*}
  where $\lambda > 0$ will be suitably chosen later.  It remains to show that
  $T$ admits a unique fixed point.  For this we will show that $T$ is a
  contraction with respect to $\left\Vert\cdot\right\Vert_{N,\lambda}$. Let
  $u,v\in C^1\left([-N,N]\right)$.  Let us denote by $K/2$ a Lipschitz
  constant for $F$.  We get
    \begin{align*}
      &\left|(Tu-Tv)'(x)e^{\Sigma(x)}\right|\\
      &\le 2\left|\int_0^x\left|F\left(y,u(y),u'(y)\right)-
          F\left(y,v(y),v'(y)\right)\right|\frac{e^{\Sigma(y)}}
        {\sigma^2(y)}dy\right|\\
      &\leq K\sup_{|z|\leq N} \frac{1}{\sigma^2(z)}
      \left | \int_0^x\left(\left| u'(y)-v'(y)\right|+\left| u(y)-v(y)\right|\right)e^{\Sigma(y)}dy\right | \\
      &\leq K\sup_{|z|\leq
        N}\frac{1}{\sigma^2(z)}\left|\int_0^xe^{\lambda|y|}dy\right|\left\Vert
        u-v\right\Vert_{N,\lambda}\\
      &=K\sup_{|z|\leq
        N}\frac{1}{\sigma^2(z)}\frac{e^{\lambda|x|}-1}{\lambda}\left\Vert
        u-v\right\Vert_{N,\lambda}.
    \end{align*}
  This implies that, for every $x \in [-N,N]$,
  \begin{equation}\label{E4}
    \left|\left(Tu-Tv\right)'(x)\right| e^{\Sigma(x)-\lambda|x|}\leq\frac{K}{\lambda}\sup_{|z|\leq N}\frac{1}{\sigma^2(z)}\left\Vert u-v\right\Vert_{N,\lambda}.
  \end{equation}
  On the other hand, since $(Tu)(0)=(Tv)(0)=x_0$ we have
    \begin{align*}
      \left|(Tu-Tv)(x)\right|&\leq\left|\int_0^x\left|\left(Tu-Tv\right)'(y)\right|dy\right|\\
      &=\left|\int_0^xe^{\Sigma(y)-\lambda|y|}\left|\left(Tu-Tv\right)'(y)\right| e^{-\Sigma(y)+\lambda|y|}dy\right|\\
      &\leq\sup_{|s|\leq N}e^{-\Sigma(s)}\sup_{|y|\leq
        N}\left(e^{\Sigma(y)-\lambda|y|}\left|(Tu-Tv)'(y)\right|\right)\frac{e^{\lambda|x|}-1}{\lambda}.
    \end{align*}
  Finally, taking into account \eqref{E4}, we get
  \begin{equation}\label{E5}
    e^{\Sigma(x)-\lambda|x|}\left|(Tu-Tv)(x)\right|\leq\frac{K}{\lambda^2}\sup_{|s|\leq N}e^{-\Sigma(s)}\sup_{|y|\leq N}e^{\Sigma(y)}\sup_{|z|\leq N}\frac{1}{\sigma^2(z)}\left\Vert u-v\right\Vert_{N,\lambda}.
  \end{equation}
  Summing up \eqref{E4} and \eqref{E5} we get
  \begin{equation}\label{E6}
    \left\Vert Tu-Tv\right\Vert_N\leq C(\lambda)\left\Vert u-v\right\Vert_{N,\lambda},
  \end{equation}
  where
  \begin{equation*}
    C(\lambda)=\frac{K}{\lambda}\sup_{|z|\leq N}\frac{1}{\sigma^2(z)}+\frac{K}{\lambda^2}\sup_{s\leq N}e^{-\Sigma(s)}\sup_{|y|\leq N}e^{\Sigma(y)}\sup_{|x|\leq N}\frac{1}{\sigma^2(x)}.
  \end{equation*}
  If $C(\lambda)<1$, \eqref{E6} has shown that $T$ is a contraction. The
  condition can be fulfilled by choosing $\lambda$ sufficiently large.
\end{proof}

\subsection{The semi-linear case with boundary conditions}
\label{sec:ellipt-probl-with}
\begin{defi}
  \label{nlb}
  Let
  \begin{enumerate}[i)]
  \item $a,b\in\R$, such that $0<a<b<\infty$,
  \item $A,B\in\R$, and
  \item $F:[a,b]\times\R^2\to\R$ be a continuous function.
  \end{enumerate}

  We say that $u:[a,b]\to\R$ of class $C^1([a,b])$ is a \emph{solution of the
    boundary value problem}

  \begin{equation}
    \label{slb}
    \left\{
      \begin{aligned}
        Lu(x)&=F(x,u,u'),\\
        u(a)&=A,\\
        u(b)&=B,
      \end{aligned}
    \right.
  \end{equation}
  if $u$ is a solution of the boundary value problem
  \begin{equation*}
    \left\{
      \begin{aligned}
        Lu&=\ell,\\
        u(a)&=A,\\
        u(b)&=B,
      \end{aligned}
    \right.
  \end{equation*}
  in the sense of Definition~\ref{DLP1} with $\ell:[a,b]\to\R$ defined by
  $\ell(x)=F\left(x,u(x),u'(x)\right)$.
\end{defi}
In Section~\ref{S5} we will observe that solving \eqref{slb} is strongly
related to the problem of solving BSDEs with random terminal time.
\begin{lemma}
  \label{lbv}
  Suppose that the assumptions of Definition~\ref{nlb} are fulfilled. Then,
  $u$ is a solution of the boundary value problem
  \begin{equation}
    \label{bp}
    \left\{
      \begin{aligned}
        Lu(x)&=F(x,u,u'),\\
        u(a)&=A\\
        u(b)&=B
      \end{aligned}
    \right.
  \end{equation}
  if and only if the functions $u_1, u_2:[a,b] \rightarrow \R$, given by
  \begin{align*}
    u_1&=u,\\
    u_2&=e^\Sigma u',
  \end{align*}
  belong to $C^1([a,b])$ and fulfill
  \begin{equation}
    \label{bp1}
    \begin{aligned}
      u_1'(x)&=e^{-\Sigma(x)}u_2(x),\\
      u_2'(x)&=2\frac{e^{\Sigma(x)}}{\sigma^2(x)}F\left(x,u_1(x),e^{-\Sigma(x)}u_2(x)\right),\\
      u_1(a)&=A,\\
      u_1(b)&=B.\\
    \end{aligned}
  \end{equation}
\end{lemma}
\begin{proof}
  Let $u$ be a solution of the boundary value problem \eqref{bp}. This means,
  by Definition~\ref{nlb}, that $u$ is a solution of the boundary value
  problem
  \begin{equation*}
    \left\{
      \begin{aligned}
        Lu&=\ell,\\
        u(a)&=A,\\
        u(b)&=B,
      \end{aligned}
    \right.
  \end{equation*}
  with
  \begin{equation*}
    \ell(x)=F\left(x,u(x),u'(x)\right),
  \end{equation*}
  in the sense of Definition~\ref{DLP1}.  By that definition,
  there are continuous extensions $\tilde u$ and $\tilde\ell$ such that
  \begin{align*}
    \tilde u|_{[a,b]}&=u,\\
    \tilde\ell|_{[a,b]}&=\ell,
  \end{align*}
  and
  \begin{equation*}
    L\tilde u =\tilde\ell
  \end{equation*}
  in the sense of Definition~\ref{D31}. Since $\tilde u\in C^1$, we can define
  \begin{equation*}
    x_a:=\tilde u'(a).
  \end{equation*}
  By Proposition~\ref{R2.8} it follows that
  \begin{equation*}
    \tilde u'(x)=e^{-\Sigma(x)}\left(2\int_a^x\frac{e^{\Sigma(y)}}{\sigma^2(y)}\tilde\ell(y)dy+x_a\right),\ \forall x\in\R.
  \end{equation*}
  By setting $\tilde \ell_1 , \tilde \ell_2, \tilde u_1, \tilde u_2: \R
  \rightarrow \R$ as
  \begin{align*}
    \tilde \ell_1&:=\tilde u',\\
    \tilde \ell_2&:=2\frac{e^\Sigma}{\sigma^2}\tilde\ell, \\
    \tilde u_1 &= \tilde u, \\
    \tilde u_2 &= \tilde u' e^\Sigma,
  \end{align*}
  it yields that $\tilde u^1, \tilde u^2$ belong to $C^1$ and
  \begin{align*}
    \tilde u_1'(x)&=\tilde\ell_1(x)\ \forall x\in\R,\\
    \tilde u_2'(x)&=\tilde\ell_2(x)\ \forall x\in\R,\\
    \tilde u_1(a)&=A,\\
    \tilde u_1(b)&=B.
  \end{align*}
  It follows now
  that $u_1,u_2\in C^1\left([a,b],\R\right)$, which are respectively
  restrictions of $\tilde u_1, \tilde u_2$, solve \eqref{bp1}.

  Concerning the converse, let $u_1,u_2\in C^1\left([a,b],\R]\right)$, so that
  \eqref{bp1} is fulfilled. We define $\tilde\ell_2: \R \rightarrow \R$ as
  \begin{equation*}
    \tilde\ell_2(x)=2\frac{e^{\Sigma(x)}}{\sigma^2(x)}\tilde\ell(x),
  \end{equation*}
  where $\tilde\ell:\R\to\R$ is a continuous extension of
  \begin{equation*}
    \ell(x)=F\left(x,u_1(x),u_2(x)e^{-\Sigma(x)}\right).
  \end{equation*}
  By \eqref{bp1}, we have for some $x_a\in\R$
  \begin{equation} \label{E340}
    u_2(x)=2\int_a^x\frac{e^{\Sigma(y)}}{\sigma^2(y)}\tilde\ell(y)dy+x_a,
  \end{equation}
  for $x \in [a,b]$.  We define $\tilde u_2: \R \rightarrow \R$ as the
  right-hand side of \eqref{E340} for all $x\in\R$. Clearly $\tilde u_2$ is a
  $C^1$ extension of $u_2$.  We also define
  \begin{equation*}
    \tilde\ell_1(x)=e^{-\Sigma(x)}\tilde u_2(x),\ \ x \in \R.
  \end{equation*}
  \eqref{bp1} gives
  \begin{equation*}
    u_1'(x) =   \tilde\ell_1(x)=e^{-\Sigma(x)}\tilde u_2(x),\ \ x\in[a,b].
  \end{equation*}
  We define $\tilde u_1 (x) = \int_a^x \tilde \ell_1(y) dy+A,\ x \in \R$.
  $\tilde u_1$ is a $C^1$ extension of $u_1$.  Consequently, setting $\tilde
  u=\tilde u_1$, we get
    \begin{align*}
      \tilde u'(x)&= \tilde u_1'(x) = \tilde \ell_1(x) = e^{-\Sigma(x)} \tilde
      u_2(x)=
      e^{-\Sigma(x)}\left(2\int_a^x\frac{e^{\Sigma(y)}}{\sigma^2(y)}\tilde\ell(y)dy+x_a\right),\\
      \tilde u(a)&=A,
    \end{align*}
  taking into account \eqref{E340} and the consideration below it. We define
  $u:[a,b] \rightarrow \R$ as restriction of $\tilde u$ and get
  \begin{align*}
    u(a) &= \tilde u(a)=A,\\
    u(b) &= \tilde u(b)=\tilde u_1(b)=u_1(b) =B,
  \end{align*}
  by \eqref{bp1}.  By Proposition~\ref{R2.8}, Definition~\ref{DLP1} and
  Definition~\ref{strongsolnonl}, $u$ is a solution to the boundary value
  problem \eqref{bp}.
\end{proof}

%
The following result provides uniqueness under some monotonicity conditions.
\begin{prop}
  \label{Mo}
  Let
    \begin{align*}
      F:[a,b]\times\R^2&\to\R,\\
      (x,y,z)&\mapsto F(x,y,z),
    \end{align*}
  be a continuous function fulfilling the following assumptions.
  \begin{enumerate}
  \item $F$ is non-decreasing in $y$, i.\,e.
    \begin{align}
      \label{mono}\left(F(x,y,z)-F(x,\tilde y,z)\right)\left(y-\tilde
        y\right)\geq 0,\ \forall y,\tilde y, z\in\R,\ x \in [a,b].
    \end{align}
  \item $F$ is globally Lipschitz (with respect to $z$).
  \end{enumerate}
  Then, for any $A,B\in\R$, the boundary value problem
  \begin{equation}
    \label{bvp}
    \left\{
      \begin{aligned}
        Lu(x)&=F\left(x,u,u'\right),\\
        u(a)&=A,\\
        u(b)&=B,
      \end{aligned}
    \right.
  \end{equation}
  has at most one $C^1$-solution.
\end{prop}
\begin{proof}
  Let $u$ and $v$ in $C^1\left([a,b]\right)$ be two solutions of the boundary
  value problem \eqref{bvp} and define
  \begin{align*}
    x_a&:=u'(a),\\
    y_a&:=v'(a).
  \end{align*}
  Then, by
  Lemma~\ref{lbv}, we get
  \begin{subequations}
    \begin{align}
      \label{u}
      u(x)&=A+\int_a^xdze^{-\Sigma(z)}\left(2\int_a^z\frac{e^{\Sigma(y)}}{\sigma^2(y)}F\left(y,u(y),u'(y)\right)dy+x_a\right),&\forall x&\in[a,b],\\
      \label{v}
      v(x)&=A+\int_a^xdze^{-\Sigma(z)}\left(2\int_a^z\frac{e^{\Sigma(y)}}{\sigma^2(y)}F\left(y,v(y),v'(y)\right)dy+y_a\right),&\forall x&\in[a,b],\\
      u(a)&=v(a)=A,\\
      u(b)&=v(b)=B.
    \end{align}
  \end{subequations}
  Indeed, we are interested in the $C^1$-function
  \begin{align*}
    \phi:[a,b]&\to \R\\
    \phi&=u-v,
  \end{align*}
  which fulfills $\phi(a)=\phi(b)=0$.
  We consider now the $C^2$-function $\chi$, given by
  \begin{subequations}
    \begin{align}
      \label{w0}\chi(a)&=0\\
      \label{w}
      \chi'(x)&=e^{\Sigma(x)}\phi'(x),
    \end{align}
  \end{subequations}
  and we define
  \begin{equation*}
    \psi:=\chi'\phi.
  \end{equation*}
  By using \eqref{w}, the monotonicity and Lipschitz conditions, we get, on
  $[a,b]$,
  \begin{multline*}
    \psi'=\chi''\phi+\chi'\phi'\geq\chi''\phi=2\frac{e^{\Sigma}}{\sigma^2}\left(F\left(x,u,u'\right)-F\left(x,v,v'\right)\right)(u-v)\\
    \geq
    2\frac{e^{\Sigma}}{\sigma^2}\left(F\left(x,\frac{u+v}{2},u'\right)-F\left(x,\frac{u+v}{2},v'\right)\right)\left(u-v\right)\geq
    -\frac{2k}{\sigma^2}\left|\chi'\phi\right|,
  \end{multline*}
  where $k$ is the Lipschitz constant.  So we get the differential inequality
  \begin{align*}
    \psi'(x)&\geq -\frac{2k}{\sigma^2(x)}\left|\psi(x)\right|,\ x\in[a,b],\\
    \psi(a)&=0,\\
    \psi(b)&=0.
  \end{align*}
  By some basic properties of differential inequalities \cite{sz} we get
  \begin{equation}
    \label{geq0}
    \psi(x)\geq 0,\ x\in[a,b].
  \end{equation}
  On the other hand,
  \begin{equation}
    \label{int}
    \int_a^b\psi(x)e^{-\Sigma(x)}dx=\int_a^b\phi'(x)\phi(x)dx=\left.\frac{\phi^2(x)}{2}\right|^b_a=0.
  \end{equation}
  Finally, combining \eqref{geq0} and \eqref{int} leads to
  \begin{equation*}
    \psi(x)=0,\ \forall x\in[a,b].
  \end{equation*}
  By definition of $\psi$ it follows that $(\phi^2)'= 0$ so that $\phi^2$ is
  constantly equal to $\phi^2(0) = 0$.

\end{proof}
We consider now a classical boundary value problem of the type considered in
\eqref{bp1}. Let $f_1, f_2:\R^3 \to\R$ be continuous and let $a,b,A,B\in\R$,
$-\infty<a<b<\infty$.  We are looking for solutions $u_1, u_2: [a,b] \to \R$
of the system
\begin{subequations}
  \label{sys}
  \begin{align}
    u_1'(x)&=f_1(x,u_1(x),u_2(x)),\label{sys1}\\
    u_2'(x)&=f_2(x,u_1(x),u_2(x)),\label{sys2}\\
    u_1(a)&=A,\\
    u_1(b)&=B.
  \end{align}
\end{subequations}
Theorem~2.1.1 in \cite{bernfeld} states the following.
\begin{theo}
  \label{123}
  Let
  $I=]\alpha,\beta]$, $-\infty\leq\alpha<\beta<\infty$, and
  $I^0=]\alpha,\beta[$. Assume the following.
  \begin{enumerate}[i)]
  \item \label{i} For every $(x,y) \in I^0\times \R $ $z \mapsto f_1(x,y,z)$
    is an increasing function.  Moreover we suppose
    \begin{equation*}
      \lim_{z\to\pm\infty}f_1(x,y,z)=\pm\infty,
    \end{equation*}
    uniformly on compact sets in $I^0\times\R$.
  \item \label{ii} All the local solutions defined on a subinterval of $I$ of
    \eqref{sys1} and \eqref{sys2} extend to a solution on the whole interval
    $I$.
  \item \label{iii} There exists at most one solution of \eqref{sys}, for all
    $a=a_0,b=b_0\in I^0 $ and all $A=A_0,B=B_0\in\R$.
  \end{enumerate}
  Then there exists exactly one solution of \eqref{sys} if $a\in I^0$ and
  $b\in I$.
\end{theo}
Previous theorem has an important consequence at the level of existence and
uniqueness of solutions to boundary value problems.
\begin{cor} \label{CEE} Let $F:[a,b]\times \R^2 \rightarrow \R$, $(x,y,z)
  \mapsto F(x,y,z)$ be a continuous function.  We suppose the following.
  \begin{enumerate}[i)]
  \item $(x,y) \mapsto F(x, y, 0) $ has linear growth with respect 
to $y$. 
  \item $F$ fulfills the monotonicity condition \eqref{mono}.
  \item $F$ is globally Lipschitz with respect to $z$.
  \end{enumerate}
  Then there exists exactly one solution to the boundary value problem
  \begin{equation}
    \label{bv}
    \left\{
      \begin{aligned}
        Lu(x)&=F(x,u,u'),\\
        u(a)&=A\\
        u(b)&=B.
      \end{aligned}
    \right.
  \end{equation}
\end{cor}
\begin{proof}
  Uniqueness follows immediately from Proposition~\ref{Mo}.  To show
  existence, we make use of Theorem~\ref{123}.  Let $\alpha < a$ and $\beta >
  b$.  We extend $F$ continuously on the entire $\R^3$ by introducing a new
  function $\tilde F$ in the following way:
  \begin{equation} \label{bv10} \tilde F(x,y,z):=
    \begin{cases}
      F(a,y,z),&x<a,\\
      F(x,y,z),&a\leq x\leq b,\\
      F(b,y,z),&x>b.
    \end{cases}
  \end{equation}
  $F$ fulfills the assumptions of Lipschitz-continuity and monotonicity, and
  so does $\tilde F$.
  At this point we can show the existence of a unique solution
  $u_1,u_2:[a,b]\to\R$ of the system
  \begin{equation}
    \label{system}
    \begin{aligned}
      u_1'(x)&=e^{\Sigma(x)}u_2(x),\\
      u_2'(x)&=2\frac{e^{\Sigma(x)}}{\sigma^2(x)}\tilde F\left(x,u_1(x),e^{-\Sigma(x)}u_2(x)\right),\\
      u_1(a)&=A,\\
      u_1(b)&=B.\\
    \end{aligned}
  \end{equation}
  That coincides with \eqref{sys} setting
  \begin{align*}
    f_1(x,y,z)&=e^{\Sigma(x)}z,\\
    f_2(x,y,z)&=2\frac{e^{\Sigma(x)}}{\sigma^2(x)}\tilde
    F\left(x,y,ze^{-\Sigma(x)}\right).
  \end{align*}
  As the mentioned existence will be a consequence of Theorem~\ref{123}, we
  check the validity of its assumptions.
  Clearly,~\ref{i} is fulfilled. Furthermore, by assumption, $\tilde
  F:\R^3\to\R$ is continuous and has linear growth
 with respect to $(y,z)$ i.e. the second and third
  variable.  Therefore assumption~\ref{ii} is fulfilled too.  Indeed, by
  Peano's theorem, we can continue (to the left and to the right) locally any
  solution of \eqref{system} to a possibly exploding solution.  The linear
  growth condition and Gronwall's lemma imply that no solution
  explodes. Moreover, Assumption~\ref{iii} of Theorem~\ref{123} holds. In
  fact, since $\tilde F$ fulfills the monotonicity condition \eqref{mono} and
  is globally Lipschitz in $z$, uniqueness follows from Proposition~\ref{Mo}.
  Finally, by Lemma~\ref{lbv}, $u=u_1$ is a solution of \eqref{bv}.
\end{proof}

The proposition below shows existence and uniqueness in the Lipschitz case
without the monotonicity condition.
\begin{prop}
  \label{pdex}
  Let $F: [0,1] \times \R^2 \to \R$, $(x,y,z) \mapsto F(x,y,z)$ be bounded and
  globally Lipschitz with respect to $(y,z)$.
  Then there exists a solution of the boundary value problem
  \begin{equation*}
    \left\{
      \begin{aligned}
        Lu(x)&=F\left(x,u,u'\right),\\
        u(0)&=A,\\
        u(1)&=B,\\
      \end{aligned}
    \right.
  \end{equation*}
  for any $A,B\in\R$.
\end{prop}
\begin{proof}
  We extend $F $ to $\tilde F$ in the way of \eqref{bv10} with $a = 0$ and $b
  =1$.  Moreover, we define a real function $\Phi:\R\to\R$ in the following
  way: for $x_0=A$ and $x_1\in\R$ we denote the solution of \eqref{PDEinit} by
  $u^{x_1}$. Its existence follows from Proposition~\ref{P34} since $\tilde F$
  is globally Lipschitz with respect to $(y,z)$. Now we set
  $\Phi(x_1)=u^{x_1}(1)$. Since $\Sigma$, $F$ and $\sigma$ are continuous,
  $\Phi$ can be shown to be continuous as well. We leave this to the reader.
  By \eqref{eq:strongsol}, we get then the following relation:
  \begin{equation}
    \label{eq:5}
    \Phi(x_1)-x_0=\int_0^1dxe^{-\Sigma(x)}\left(2\int_0^x\frac{e^{\Sigma(y)}}{\sigma^2(y)}
      F\left(y,u^{x_1}(y),\left(u^{x_1}\right)'(y)\right)dy+x_1\right).
  \end{equation}
  Since $F$ is bounded,
  \begin{equation}
    \label{eq:6}
    \lim_{x_1\to\infty}\Phi(x_1)=\infty=-\lim_{x_1\to -\infty}\Phi(x_1).
  \end{equation}
  Consequently, by mean value theorem, for each $B\in\R$, there is an $x_1$ so
  that $\Phi(x_1)=B$.
\end{proof}
\begin{rem} \label{R313} If $F$ is not bounded, one cannot ensure existence in
  general. To give an example, we set $L=\frac{d^2}{dx^2}$ and
  $F(x,y,z)=-\pi^2y$. Then the corresponding boundary value problem
  \begin{equation*}
    \left\{
      \begin{aligned}
        u''&=-\pi^2 u,\\
        u(0)&=0,\\
        u(1)&=1,
      \end{aligned}
    \right.
  \end{equation*}
  has no solution.
\end{rem}
\begin{prop}
  \label{FP}
  Let $a=0$, $b=1$, and $F: [0,1] \times \R \times \R \rightarrow \R$,
  $(x,y,z)\mapsto F(x,y,z)$ be globally Lipschitz with respect to $(y,z)$ and
  Lipschitz-constant $k$, fulfilling
  \begin{equation}
    \label{CondM}
    k <\left(\sup_{x\in[0,1]}\int_0^1dy\left(\left|K(x,y)\right|+\left|\partial_xK(x,y)\right|\right)\right)^{-1},
  \end{equation}
  where $K$ was defined in \eqref{exp3}.  Then, \eqref{slb} has a unique
  solution for any $A,B\in\R$.
\end{prop}
\begin{proof}
  We consider the map $T:C^1([0,1])\to C^1([0,1])$ defined by
  \begin{equation*}
    Th(x)=f(x)+\int_0^1K(x,y)F(y,h(y),h'(y))dy,
  \end{equation*}
  with $f$
  is given by \eqref{exp2}.
  Taking into account Definition~\ref{nlb} and Proposition~\ref{P11},
  \eqref{slb} is well-posed if and only if $T$ has a fixed point.  We show the
  latter assertion.  $C^1([0,1])$ is a Banach space equipped with the norm
  \begin{equation*}
    \Vert h\Vert=\sup_{x\in[0,1]}\left(\left|h(x)\right|+\left|h'(x)\right|\right).
  \end{equation*}
  To show that $T$ admits a unique fixed point, we will show that $T$ is a
  contraction with respect to $\Vert\cdot\Vert$. Let $u,v\in C^1([0,1])$. We
  get
    \begin{align}
      \left|(Tu-Tv)(x)\right|&=\left|\int_0^1K(x,y)\left(F\left(y,u(y),u'(y)\right)-F\left(y,v(y),v'(y)\right)\right)dy\right|\nonumber\\
      &\leq\int_0^1dy\left|K(x,y)\right| k \left(\left|u(y)-v(y)\right|+\left|u'(y)-v'(y)\right|\right)\nonumber\\
      &\leq\int_0^1dy\left|K(x,y)\right|k \Vert u-v\Vert\label{c1}
    \end{align}
  and
    \begin{align}
      \left|\left(Tu-Tv\right)'(x)\right| &= \left|\int_0^1
        \partial_x K(x,y)\left(F\left(y,u(y),u'(y)\right)-F\left(y,v(y),v'(y)\right)\right)dy\right| \nonumber\\
      &\le \left|\int_0^1dy\partial_xK(x,y)k \left(\left|u(y)-v(y)\right|+\left|u'(y)-v'(y)\right|\right)dy\right|\nonumber\\
      &\leq \int_0^1dy\left|\partial_xK(x,y)\right|k \Vert u-v\Vert.\label{c2}
    \end{align}
  Summing up \eqref{c1} and \eqref{c2} and taking the supremum over $x$ gives
  \begin{equation*}
    \left\Vert Tu-Tv\right\Vert\leq\sup_{x\in[0,1]}\int_0^1dy\left(\left|K(x,y)\right|+\left|\partial_xK(x,y)\right|\right)k\Vert u-y\Vert.
  \end{equation*}
  It follows that $T$ is a contraction if $k$ fulfills \eqref{CondM}.
\end{proof}
\section{Exit time of the solution to the forward martingale problem}
\label{S4}
We are interested in the nature of the first exit time $\tau$ from the
interval $[0,1]$ of a solution $X=X^x$ to the martingale problem with respect
to $L$ and initial condition $x\in[0,1]$. So we define $\tau$ as
\begin{equation*}
  \tau:= \left \{\begin{aligned}
      &\inf\left\{t\geq 0\big| X_t\notin[0,1]\right\},&\text{if}\  
      &\left\{t\geq 0\big| X_t\notin[0,1]\right\}\neq\emptyset\\
      &\infty,&&\text{otherwise.}
    \end{aligned} \right.
\end{equation*}
\begin{prop}
  \label{P200} $\tau$ has finite expectation. In particular $\tau$ is finite
  almost surely.
\end{prop}
\begin{proof}
  We consider $\Gamma:[0,1]\to\R$ as the unique solution of
    \begin{align*}
      L\Gamma&=-1\\
      \Gamma(0)&=\Gamma(1)=0,
    \end{align*}
  and an extension $\tilde \Gamma\in\shd_L$ as regarded in
  Definition~\ref{DLP1}. Since $X$ is a solution to the martingale problem
  with respect to $L$ and initial condition $x$, the process
  \begin{equation*}
    N_t=\tilde \Gamma(X_t)-\tilde \Gamma(x)-\int_0^tL\tilde \Gamma(X_r)dr,
  \end{equation*}
  is a local martingale.
  By Proposition~\ref{P48} we have $N_t = \int_0^t \tilde \Gamma'(X_s)
  dM^X_s$, which, by Remark~\ref{RR} \ref{RR1}, implies that
  \begin{equation*} [N]_t=\int_0^t\sigma^2(X_s)\tilde \Gamma'(X_s)^2ds.
  \end{equation*}
  Now, let $\left(\tau_n\right)$ be the family of stopping times defined as
  \begin{equation*}
    \tau_n:=\inf\left\{t\geq 0\left|\int_0^t\right.\sigma^2(X_s)\tilde
      \Gamma'(X_s)^2ds\geq n\right\},
  \end{equation*}
  with the assumption that $\inf\left(\emptyset\right)=\infty$.

  The stopped processes $N^{\tau_n}$ are clearly square integrable
  martingales. By Doob's stopping theorem for martingales, the processes
  $\left(N^{\tau_n}_{t\wedge\tau}\right)_{t\geq 0}$ are again
  martingales. Consequently,
  \begin{equation*}
    E\left(\tilde \Gamma \left(X_{\tau_n\wedge t\wedge\tau}\right)-
      \tilde \Gamma(x)-\int_0^{\tau_n\wedge t\wedge\tau}\left(L\tilde \Gamma\right)
      (X_r)dr\right)=0.
  \end{equation*}
  Since $L\tilde \Gamma$ restricted to $[0,1]$ equals $-1$, the previous
  expression gives
  \begin{equation*}
    E\left(\tilde \Gamma\left(X_{\tau_n\wedge t\wedge\tau}\right)-
      \tilde \Gamma(x)\right)+E\left(\tau_n\wedge t\wedge\tau\right)=0.
  \end{equation*}
  Now we take the limit $n\to\infty$, and we can use the theorems of monotone
  and dominated convergence, since
  \begin{equation*}
    \left|\tilde \Gamma \left(X_{\tau_n\wedge t\wedge\tau}\right)\right|\leq\sup_{x\in[0,1]}\left| \Gamma(x)\right|.
  \end{equation*}
  This gives, for every $x\in[0,1]$,
  \begin{equation}
    \label{E40}
    E\left(\Gamma \left(X_{t\wedge\tau}\right)\right)-
    \Gamma(x)+E\left(t\wedge\tau\right)=0.
  \end{equation}
  Finally, letting $t\to\infty$, we get
  \begin{equation*}
    E\left(\tau\right)=\Gamma(x)-E\left(\Gamma\left(X_\tau\right)\right)
    =\Gamma(x),
  \end{equation*}
  by the same arguments as above taking $n \rightarrow \infty$.
\end{proof}
As byproduct of the proof of Proposition~\ref{P200} we get the following.
\begin{prop}
  \label{C200}
  The expectation of the exit time $\tau$ is exactly $\Gamma(x)$, where
  $\Gamma$ is the unique solution of
    \begin{align*}
      L\Gamma&=-1\\
      \Gamma(0)&=\Gamma(1)=0.
    \end{align*}
\end{prop}

\section{Martingale driven BSDEs with random terminal time}
\label{S5}

%

\subsection{Notion of solution}

\label{S51}

The present section does not aim at the greatest generality, which could be
the object of future research. We consider the case of one-dimensional BSDEs
driven by square integrable martingales with continuous predictable bracket.



Backward SDEs driven by martingales were investigated by several authors, see
e.\,g.\ \cite{buckdahn}, \cite{sant}, see also \cite{elkarouihuang},
\cite{buen1} and \cite{CCR} for recent developments.  We are interested in
such a BSDE with terminal condition at random time.  This is motivated by the
fact that the forward SDE (martingale problem) only admits weak solutions,
therefore the reference filtration will only be the canonical one related to
the solution and not the one associated with the underlying Brownian
motion. We consider the following data.
\begin{enumerate}[i)]
\item \label{st}An a.\,s.\ finite stopping time $\tau$.
\item \label{mc} An $\shf$-local martingale $\left(M_t\right)_{t\geq 0}$
  with an $\shf$-predictable continuous quadratic variation process
  $\left<M\right>$.  We suppose moreover that $M^\tau$ is an $\shf$-square
  integrable martingale, and we suppose the existence of a deterministic
  increasing function $\rho: \R_+ \rightarrow \R_+$ with $\rho (0) = 0$ and
$$ \left<M^\tau \right>_t  \le \rho(t),\ \forall t \ge 0.
  $$
\item \label{tc}A terminal condition $\xi\in
  L^2\left(\Omega,\shf_\tau,P;\mathbb R\right)$.
\item \label{pr}A coefficient $f:\Omega\times[0,T]\times\R^2\to\R$, such that
  the process $f(\cdot,t,y,z)$, $t\geq 0$, is predictable for every $y,z$.
\end{enumerate}
\begin{defi}
  \label{D1}
  Let $(Y,Z,O)$ be a triple of processes with the following properties.
  \begin{enumerate}[i)]
  \item \label{D1i}$Y$ is c\`adl\`ag
    $\shf$-adapted.
  \item \label{D1ii}$Z$ is $\shf$-predictable such that $E\left(\int_0^\tau
      Z_s^2d\left<M\right>_s\right)<\infty$.
  \item \label{D1iii}$O$ is a square integrable martingale such that $O_0=0$
    and $E\left(O^2_\tau\right)<\infty$. Furthermore, $O$ is strongly
    orthogonal to $M$, i.\,e.\ $\left< M, O \right> = 0$.
  \item \label{D1iv} $Z_t=0$ if $t>\tau$ and $O_t = O_\tau$ for $t \ge \tau$.
  \end{enumerate}
  Such a triplet $(Y,Z,O)$ is called \emph{solution} of the BSDE
  $(f,\tau,\xi)$ if it fulfills
  \begin{equation}
    \label{1}
    Y_t=\xi-\int_t^\infty\1_{\left\{\tau\geq s\right\}}Z_sdM_s+
    \int_t^\infty\1_{\{\tau\geq s\}}f\left(\omega,s,Y_s,Z_s\right)d\left<M\right>_s-\left(O_\tau-O_{t\wedge\tau}\right).
  \end{equation}
\end{defi}
\begin{rem} \label{R53} \
  \begin{enumerate}[i)]
  \item If $t\geq\tau$ in \eqref{1} we get $Y_t=\xi=Y_\tau$, so in particular
    $Y_t=Y_\tau,\ t\geq\tau$.
  \item Indeed we will always suppose that $M =M^\tau$ so that \eqref{1} can
    be rewritten as
    \begin{equation}
      \label{1bis}
      Y_t=\xi-\int_t^\infty Z_sdM_s+
      \int_t^\infty f\left(\omega,s,Y_s,Z_s\right)d\left<M\right>_s-
      \left(O_\tau-O_{t\wedge\tau}\right).
    \end{equation}
  \end{enumerate}
\end{rem}


When $M$ is a Brownian motion, this was treated in \cite{dapa} from which we
inherit and adopt very close notations.

\subsection{Uniqueness of Solutions}
\label{S53}

\begin{theo}\label{thu} Let $a, b,\kappa\in \R$ and set  $\gamma=b^2-2a$.
  We suppose the following.
  \begin{enumerate}[i)]
  \item
    $\left(f\left(\omega,s,y_1,z\right)-f\left(\omega,s,y_2,z\right)\right)\left(y_1-y_2\right)\leq
    -a\left|y_1-y_2\right|^2$,\label{as1} for every $\omega\in\Omega$, $s \in
    [0,T]$, $y_1,y_2,z \in \R$.
  \item
    $\left|f\left(\omega,s,y,z_1\right)-f\left(\omega,s,y,z_2\right)\right|\leq
    b\left|z_1-z_2\right|$,\label{as2} for every $\omega\in\Omega$, $s \in
    [0,T]$, $y \in \R$.
  \item $\left|f(\omega,s,y,z)-f(\omega, s,0,0)\right|\leq\kappa \left(|y|+
      \kappa' \right)+b|z|$,\label{as3} where $\kappa' \in\{1,0\}$.
  \item
    $E\left(\int_0^\tau e^{\theta \left<M\right>_t}\left(f(t,0,0)^2 +
        \kappa'\right) d\left<M\right>_t \right)<\infty$ for every $\theta <
    \gamma$.
  \end{enumerate}

  Let $\xi\in L^2\left(\Omega,\shf_\tau\right)$. Then the BSDE $(f,\tau,\xi)$
  admits at most one solution $(Y,Z,O)$, such that
  \begin{equation} \label{sb} E\left(Y_0^2 + \int_0^\tau
      e^{\gamma\left<M\right>_t}(Y^2_t + Z_t^2) d \left< M \right>_t
      +e^{\gamma \left<M\right>_t}d\left<O\right>_t\right)<\infty.
  \end{equation}
\end{theo}

\begin{rem}\label{RY0}
  \
  \begin{enumerate}
  \item In the proof of Theorem~\ref{thu} in appendix~\ref{ProofTheorem}, we
    omit the dependence of $f$ on $\omega$ in order to simplify the notations.
  \item If we suppose $\shf_0$ to be the trivial $\sigma$-field, then $Y_0^2$
    can be deleted in \eqref{sb}.
  \end{enumerate}
\end{rem}
In the proof of Theorem~\ref{thu} in appendix~\ref{ProofTheorem}, we use the
following technical lemma, which is the generalization of Proposition 4.3 in
\cite{dapa}.
\begin{lemma}
  \label{L1}
  Suppose the validity of hypotheses~\ref{as1},~\ref{as2} and~\ref{as3} of
  Theorem~\ref{thu},
  and let $(Y,Z,O)$ be a solution of BSDE $(f,\tau,\xi)$ such that for some
  $\theta$,
  \begin{equation}
    \label{ic}
    E\left(Y_0^2 + \int_0^\tau e^{\theta\left<M\right>_s}\left(\left|Y_s\right|^2+\left|Z_s\right|^2 +  f^2(s,0,0) + \kappa'  \right)d\left<M\right>_s
      + \int_0^\tau e^{\theta \left< M \right>_s} d \left< O \right>_s
    \right)<\infty.
  \end{equation}
  Then
  \begin{align}
    \label{E1}
    E\left(\sup_{s\leq\tau}e^{\theta\left<M\right>_s}\left|Y_s\right|^2\right)
    <\infty,
  \end{align}
  and
  \begin{equation}
    \label{E2}
    N_t=\int_0^{t\wedge\tau}e^{\theta\left<M\right>_s}Y_{s-}\left(Z_sdM_s+dO_s\right)=\int_0^{t}e^{\theta\left<M\right>_s}Y_{s-}\left(Z_sdM_s+dO_s\right)
  \end{equation}
  is a uniformly integrable martingale.
\end{lemma}
We prove this lemma in appendix~\ref{ProofLemma}.

\begin{rem} \label{RExistence} Adapting the results of \cite{dapa} Proposition
  3.3, it is possible to state and prove also an existence theorem. We have
  decided not to do it for two reasons.
  \begin{enumerate}
  \item The techniques can be adapted from the proof of Proposition 3.3 by the
    same techniques as in the proof of Theorem~\ref{thu}.
  \item For our applications to the probabilistic representation of semilinear
    PDEs, we already provide an \emph{existence} theorem through the
    resolution of the PDE.
  \end{enumerate}
\end{rem}

\section{Solutions for BSDEs via solutions of elliptic PDEs}
\label{sec:existence}

In this final section we will make the assumptions of Section~\ref{S22} which
guarantee existence and uniqueness in law of the martingale problem with
respect to $L$.  In particular we will suppose that $\sigma > 0$, $\Sigma$ as
defined in \eqref{ESigma} exists and we assume the validity
 of \eqref{F2} for the function
$v$ defined in \eqref{F1}.

Let $x_0 \in \R$ and let $X$ solve a martingale problem $\MP(\sigma, \beta; x_0)$ \eqref{SFSDE}.
We are interested in a BSDE with terminal condition at the
random time $\tau$, which is the exit time of $X$ from interval $[0,1]$.
In this section $\shf$ is  the canonical filtration $\shf^X$ of $X$.
Since $X$ solves the martingale problem, by Remark~\ref{RR}, $X$ is an
$\shf^X$-Dirichlet process. From now on  its  $\shf^X$-local martingale 
component $M^X$
will also be denoted by $M$, in agreement
with Section \ref{S5}.

Let $F:\R^3\to\R$ be a  continuous function. We set
\begin{align}
  \label{eq:7}
  f(\omega,t,y,z)&=- \frac{F(X_t(\omega),y,z)}{\sigma^2(X_t(\omega))},
&t&\geq 0,\ y,z\in\R,
\end{align}
and we define
\begin{equation} \label{EETau} \tau=\inf\left\{\left.t\geq 0\right| X_t\notin
    I\right\}.
\end{equation}
Let $u_0, u_1 \in \R$ and set
$\xi=\1_{\left\{X_\tau=0\right\}}u_0+\1_{\left\{X_\tau=1\right\}}u_1$. So
\begin{equation} \label{EXi} \xi=u(X_\tau),
\end{equation}
for a function $u:[0,1] \rightarrow \R$ such that $u(0) = u_0, u(1) = u_1$.

Our method allows to construct solutions of the BSDE $(f,\xi,\tau)$ even in cases
that $f$ does not necessarily fulfill Lipschitz or monotonicity assumptions.

We need to check that we are in the framework of the hypotheses at the
beginning of Section~\ref{S51}.
\begin{itemize}
\item~\ref{st} is verified because of Proposition~\ref{P200}.
\item~\ref{mc} holds because
  \begin{equation*}
    \left<M^\tau \right>_t = \int_0^{t \wedge \tau} \sigma^2(X_s) ds  \le \rho(t),
  \end{equation*}
  where $\rho(t) = t \sup_{x\in [0,1]} \sigma^2(x)$.

\item~\ref{tc} is fulfilled since $\xi$ is a bounded random variable, of
  course $\shf_\tau$-measurable.
\item~\ref{pr} is verified by construction, and because $X$ is a continuous
  adapted process.
\end{itemize}
The aim of this section is to show that the $C^1$-type solutions of elliptic
PDEs in the sense of Definition~\ref{nlb} produce solutions to a BSDE of the
type defined in Definition~\ref{D1}.

\begin{rem} \label{R56} \
  \begin{enumerate}
  \item $\mathcal F^X$ is generally not a Brownian filtration, so that the
    theory of \cite{dapa} for existence and uniqueness of BSDEs with random
    terminal time cannot directly be applied.
  \item Even for a simple equation of the type
    \begin{equation*}
      dX_t=\sigma_0(X_t)dW_t,
    \end{equation*}
    where $\sigma_0$ is only a continuous bounded non-degenerate function,
    $\shf^X$ is not necessarily equal to $\shf^W$ even though $W$ is an
    $\shf^X$-Brownian motion.
  \item In general, the solution of a semilinear differential equation of the
    type \eqref{slb} can be associated with the solution of a BSDE driven by
    the martingale $M^X$ which is the martingale component of the
    $\shf^X$-Dirichlet process $X$.
  \item In Section~\ref{S5} we have investigated BSDEs driven by (even not
    continuous) martingales, which are of independent interest.
  \end{enumerate}

\end{rem}

\begin{theo}\label{PB4}
  Let $I=\left[0,1\right]$ and $u:I\to\R$ be a $C^1$-solution of
  \begin{equation}
    \label{PDE}
    \left\{
    \begin{aligned}
      Lu\left(x\right)&=F\left(x,u\left(x\right),u'\left(x\right)\right)\\
      u\left(0\right)&=u_0\\
      u\left(1\right)&=u_1.
    \end{aligned}
    \right.
  \end{equation}

 Let $(X_t) = X^{x_0}$ be a solution of
  $\MP\left(\sigma,\beta;x_0\right)$ on some probability space
  $\left(\Omega,\mathcal{G},P\right)$.
  We set, for $t \in [0,T]$,
  \begin{align*}
    \label{6.23bis}
    Y_t&=u(X_t^\tau)\\
    Z_t&=u'(X_t) \1_{[0,\tau]}(t)\\
    O_t &= 0. \\
  \end{align*}
  Then $\left(Y,Z, O\right)$ is a solution on $\left(\Omega,\mathcal{G},
    P\right)$ to the BSDE $\left(f,\xi,\tau\right)$, where $f, \tau, \xi$ were
  defined in \eqref{eq:7}, \eqref{EETau}, \eqref{EXi}.

\end{theo}
\begin{proof}
  We recall that, by Proposition~\ref{P200}, $\tau<\infty$ almost surely.

  By Definitions~\ref{nlb} and~\ref{DLP1}, there exists $\tilde u \in\mathcal
  D_L$ which extends $u$ to the real line and $L \tilde u = \tilde \ell$ and
  $\tilde \ell: \R \rightarrow \R$ is a continuous function extending $\ell(x)
  = F(x,u(x),u'(x))$.  By the definition of the martingale problem,
  \begin{equation}\label{BBB}
    M_t^{\tilde u}:=\tilde u(X_t)- \tilde u(X_0)-\int_0^t L \tilde u(X_s)ds, \
    t  \in [0,T],
  \end{equation}
  is an $\mathcal{F}^X$-local martingale.
	
  By Remark~\ref{R1.3} $\tilde Y_t= \tilde u(X_t)$ is an
  $\left(\shf^X\right)$-Dirichlet process with martingale component
  $\int_0^t {\tilde u}'(X_s)dM^X_s$. On the other hand, by \eqref{BBB},
  $\tilde Y$ is an $\left(\shf^X\right)$-semimartingale with
  martingale component $M^{\tilde u} $.
  By uniqueness of decomposition of Dirichlet processes 
	$$ M^{\tilde u}_t = \int_0^t \tilde u'(X_s)dM^X_s.$$
	We set  now
    \begin{align*}
      Y_t&=\tilde u(X_{t\wedge\tau})\\
    Z_t&= \tilde u'(X_t) \1_{[0,\tau]}(t).
 \end{align*}
  \eqref{BBB} gives
  \begin{equation*}
    \tilde u(X_t)- \tilde u(X_0) = \int_0^tL \tilde u(X_s)ds+
    \int_0^t \tilde u'(X_s)dM^X_s.
  \end{equation*}
  Stopping previous identity at time $\tau$ implies for every $T>0$ that
  \begin{equation*}
    Y_{T\wedge\tau}-Y_{t\wedge\tau}=\int_{t\wedge\tau}^{T\wedge\tau}L \tilde u(X_s)ds+
   \int_{t\wedge\tau}^{T\wedge\tau} u'(X_s) dM^X_s.
  \end{equation*}
  Letting $T\to\infty$, since $\tau<\infty$ a.\,s.\ gives
  \begin{equation*}
    Y_t=Y_\tau-\int_{t\wedge\tau}^\tau L \tilde u(X_s)ds
-\int_{t\wedge\tau}^\tau Z_sdM_s^X.
  \end{equation*}
 
So $(Y,Z,O)$ solves BSDE $(f,\tau,\xi)$
  with $O\equiv 0$.  In particular the conditions of Definition~\ref{D1} are
  fulfilled.  In fact~\ref{D1i},~\ref{D1iii} and~\ref{D1iv} are
  trivial.~\ref{D1ii} holds since $\tilde u'$ is bounded on $[0,1]$.
Moreover \eqref{1} is fulfilled since $u$ solves \eqref{PDE},
taking into account \eqref{eq:7}.
 
\end{proof}
\begin{rem} \label{RPB4} Since $u$ is bounded, we also have
  $E\left(\sup_{t\leq\tau} Y_t^2\right)<\infty$.
\end{rem}

By Theorem~\ref{PB4}, Corollary~\ref{CEE} and the Propositions~\ref{pdex}
and~\ref{FP}, we conclude the following.
\begin{cor}\label{C58a}
  Let $F:[0,1]\times\R^2\to\R$ be continuous. Suppose that at least one of the
  following assumptions holds.
  \begin{enumerate}[a)]
  \item $(x,y) \mapsto  F(x,y,0)$ has linear growth with respect to $y$, 
\eqref{mono} is fulfilled and $F$ is globally Lipschitz in $z$.
  \item $(x,y,z) \mapsto F(x,y,z)$ is bounded and globally Lipschitz with respect to $(y,z)$.
  \item $(x,y,z) \mapsto F(x,y,z)$ is globally Lipschitz with respect to
 $(y,z)$ and
    Lipschitz-constant $k$, fulfilling
    \begin{equation*}
      k<\left(\sup_{x\in[0,1]}\int_0^1dy\left(\left|K(x,y)\right|+\left|\partial_xK(x,y)\right|\right)\right)^{-1},
    \end{equation*}
  \end{enumerate}
  $K$ being the kernel introduced in \eqref{exp3}. Then there is a solution
  $(Y,Z,O)$ of BSDE $(f,\tau,\xi)$, given by \eqref{1}, where $f,\tau,\xi$
  were defined in \eqref{eq:7}, \eqref{EETau}, \eqref{EXi}.
\end{cor}
\begin{rem} \label{RC58a} The solution is provided in the statement of
  Theorem~\ref{PB4}.
\end{rem}
Corollary \ref{C59} follows from Corollary \ref{C58a} and Theorem
\ref{thu}.
\begin{cor}
  \label{C59}
  Let $F:[0,1]\times\R^2\to\R$ with the following assumptions.
  \begin{enumerate}[i)]
  \item $(x,y) \mapsto  F(x,y,0)$ has linear growth in $y$, 
  \item
    $\left(F\left(x,y_1,z\right)-F\left(x,y_2,z\right)\right)\left(y_1-y_2\right)\geq
    a\left(y_1-y_2\right)^2$ for some $a$,
  \item $F$ is globally Lipschitz in $z$ with constant $b$.
  \item $\gamma=b^2-2a\leq 0$.\label{C594}
  \end{enumerate}

  Then the solution $(Y,Z,O)$ provided by Corollary~\ref{C58a} is unique in
  the class of
  \begin{equation}
    \label{E59}
    E\left(\int_0^\tau e^{\gamma\left<M\right>_s}Y^2_sd\left<M\right>_s+\int_0^\tau e^{\gamma\left<M\right>_s}Z^2_sd\left<M\right>_s+\int_0^\tau e^{\gamma\left<M\right>_s}d\left<O\right>_s\right)<\infty.
  \end{equation}
\end{cor}
\begin{rem}\label{R512}
  \
  \begin{enumerate}[a)]
  \item Condition iv) implies that $a > 0$. In particular $F$ is increasing in
    $y$.
  \item The validity of hypotheses i), ii), iii) imply Hypothesis a) in
    Corollary~\ref{C58a}.
  \item The solution provided by Corollary~\ref{C58a} fulfills \eqref{E59}
    since $\gamma\le0$, $u$, $u'$ are bounded and $O \equiv 0$,
taking into account that $E(\tau) < \infty$, by Proposition \ref{P200}.

    \end{enumerate}
\end{rem}
\begin{rem} \label{R60} We discuss an example related to the case $\gamma$
  strictly positive, i.e.\ when Assumption~\ref{C594} of Corollary \ref{C59}
is not fulfilled.
   Consider $F(x,y,z)=-\pi^2y$.
  \begin{enumerate}
  \item The PDE
    \begin{equation}
      \label{E61}
      \begin{cases}
        u''(x)=F\left(x,u(x),u'(x)\right)\\
        u(0)=u(1)=0,
      \end{cases}
    \end{equation}
    is not well-posed, since $u(x)=\eta \sin(\pi x), \eta \in \R$, provide
    a class of solutions of \eqref{E61}, and so Theorem~\ref{PB4} provides a
    family of solutions of BSDE $(f,\tau,\xi)$, $\xi\equiv 0$, $f, \tau, \xi$
    being defined in \eqref{eq:7}, \eqref{EETau}, \eqref{EXi}, when $X_t=
    \frac{1}{2} + W_t$ and $W$ is a standard Brownian motion. In particular
    $X$ solves $\MP\left(\sigma,\beta;x_0\right)$ in the sense of Definition
\ref{DMartPr} 
with $x_0 = \frac{1}{2}$, $\sigma = 1, \beta =    0$.  We remark that
    $a=-\pi^2, b = 0$, so $\gamma = b^2-2a = 2 \pi^2 > 0$. \\

  \item In that case, since the mentioned
    Assumption~\ref{C594} is not fulfilled, then of course Corollary~\ref{C59}
    cannot be applied.  Indeed, if 
$\eta \neq 0$,
 the solutions,  
 provided explicitly above are
    not in the class of solutions fulfilling \eqref{E59}, 
 as we show below.

    Let $\tau$ be the exit time of Brownian motion $X$ starting from $x_0 =
    \frac{1}{2}$ from interval $[0,1]$. 
By Proposition \ref{PC1} in appendix~\ref{moments} we have
  $E\left(\exp(
    \gamma \tau)\right) = \frac{1}{\cos(\sqrt{\frac{\gamma}{ 2}})}$ if $ 0 \le \gamma < \frac{\pi^2}{2}$, and
 $E\left(\exp( \gamma \tau)\right) =
    \infty$ whenever $\gamma \ge  \frac{\pi^2}{2}$.
    Consequently if $\gamma = 2 \pi^2$ as before, then $E\left(\exp( \gamma \tau)\right)
    = \infty$.

    We have $\langle M \rangle_t \equiv t, O \equiv 0$.  So,  \eqref{E59} gives
\begin{multline*}
 E\left( \int_0^\tau ds e^{\gamma s} \eta^2 \left(\sin^2\left(\pi X_s\right)
  + \pi^2 \cos^2\left(\pi X_s\right)\right)\right) \\
  \ge E \left(\int_0^\tau e ^{\gamma s}
  \eta^2 ds\right) = \frac{\eta^2}{\gamma} E\left(e^{\gamma \tau}\right) = \infty,
\end{multline*}
since $\eta \neq 0$. 
 In conclusion, if
$\gamma >0$, the solutions provided by Corollary~\ref{C58a} may  fulfill
or not \eqref{E59}.
\item On the other hand the solutions above fulfill the version of \eqref{E59}
  with $\gamma= 0$, i.e.
  \begin{equation} \label{s0} E\left( \int_0^\tau (Y^2_t + Z_t^2) d \left< M
      \right>_t +d\left<O\right>_t\right)<\infty,
  \end{equation}
  since $\tau$ has finite expectation and $u, u'$ are bounded. In particular
  the class of solutions fulfilling only \eqref{s0} is not, in general, a good
  class for uniqueness.
\end{enumerate}
\end{rem}
\section*{Appendix}
\appendix
\section{Proof of Lemma~\ref{L1}}
\label{ProofLemma}
Since $Y$ solves the BSDE, by integration by parts we get
\begin{multline}
  \label{20}
  e^{\frac{\theta}{2}\left<M\right>_{t\wedge\tau}}Y_{t\wedge\tau}=Y_0+\int_0^{t\wedge\tau}e^{\frac{\theta}{2}\left<M\right>_s}\left(Z_sdM_s+dO_s\right)\\
  -\int_0^{t\wedge\tau}e^{\frac{\theta}{2}\left<M\right>_s}f\left(s,Y_s,Z_s\right)d\left<M\right>_s+\frac{\theta}{2}\int_0^{t\wedge\tau}e^{\frac{\theta}{2}\left<M\right>_s}Y_sd\left<M\right>_s.
\end{multline}
By Assumption~\ref{mc} of Section~\ref{S51}, $\left<M\right>$ is
continuous. Consequently,
\begin{equation}
  \label{21}
  \left[e^{\frac{\theta}{2}\left<M\right>}Y\right]_{t\wedge\tau}=\left[N^\theta\right]_{t\wedge\tau},
\end{equation}
where
\begin{equation} \label{E21}
  N_t^\theta:=\int_0^te^{\frac{\theta}{2}\left<M\right>_s}\left(Z_sdM_s+dO_s\right).
\end{equation}
\begin{rem}
  \label{R1}
  From \eqref{ic} it follows that
  \begin{equation*}
    E\left(\int_0^\tau e^{\theta\left<M\right>_s}\left(Z_s^2d\left<M\right>_s+d\left<O\right>_s\right)\right)<\infty.
  \end{equation*}
  Consequently, $N^\theta$ is a square integrable martingale. So, by the proof
  of Proposition 4.50 in \cite{jacod}, there is a uniformly integrable
  martingale $\shm^\theta$, so that
  \begin{equation*}
    \left[N^\theta\right]=\left<N^\theta\right>+\shm^\theta.
  \end{equation*}
\end{rem}
We continue with the proof of Lemma~\ref{L1} by using It\^{o}'s formula and
\eqref{20} getting
  \begin{align}
    e^{\theta\left<M\right>_{t\wedge\tau}}Y^2_{t\wedge\tau}-Y_0^2=&\left(e^{\frac{\theta}{2}
        \left<M\right>_{t\wedge\tau}}Y_{t\wedge\tau}\right)^2-Y_0^2\nonumber\\
    =\ &2\int_0^{t\wedge\tau}
    e^{\frac{\theta}{2}\left<M\right>_{s}}Y_{s-}d\left(e^
      {\frac{\theta}{2}\left<M\right>_{s}}Y_{s}\right)+\left[e^{\frac{\theta}{2}\left<M\right>}Y\right]_{t\wedge\tau}\nonumber\\
    =\ &2\int_0^{t\wedge\tau}e^{\theta\left<M\right>_s}Y_{s-}\left(Z_sdM_s+dO_s\right)\nonumber\\
    &-2\int_0^{t\wedge\tau}e^{\theta\left<M\right>_s}Y_s f\left(s,Y_s,Z_s\right)d\left<M\right>_s\nonumber\\
    &+2\frac{\theta}{2}\int_0^{t\wedge\tau}e^{\theta\left<M\right>_s}
    Y_s^2d\left<M\right>_s+ \left[N^\theta\right]_{t\wedge\tau},\label{E510}
  \end{align}
where in the latter equality we have taken into account \eqref{21}.  Since
$\left< M \right>$ is continuous we have been allowed to replace $Y_{s-}$ with
$Y_s$ in the two lines above.  By use of Cauchy-Schwarz, the inequality $ 2
\alpha \beta \le \alpha^2 + \beta^2$ and assumption~\ref{as3} of
Theorem~\ref{thu}, there is a constant $c$, depending on $\kappa$, $b$ and
$\theta$, such that
\begin{multline}
  \label{12}
  e^{\theta\left<M\right>_{t\wedge\tau}}Y^2_{t\wedge\tau}-Y^2_0 \leq
  c\int_0^{t\wedge\tau}e^{\theta\left<M\right>_s}\left( Y^2_{s}+Z_s^2\right.
  \\+\left.f^2(s,0,0)+\kappa'\right)d\left<M\right>_s+2\int_0^{t\wedge\tau}e^{\frac{\theta}{2}\left<M\right>_s}Y_{s-}dN_s^\theta+\left[N^\theta\right]_{t\wedge\tau}.
\end{multline}
Now we continue with a localization of \eqref{12}. For that we define for each
$n\in\N$ a stopping time $\tau(n)$ by
\begin{align*}
  \tau(n)&:=\inf\left\{t\vert Y_t\geq n\right\}\wedge n.
\end{align*}
Replacing $t$ with $t\wedge\tau(n)$ in \eqref{12} gives
\begin{multline}
  \label{13}
  e^{\theta\left<M\right>_{t\wedge\tau(n)\wedge\tau}}Y^2_{t\wedge\tau(n)\wedge\tau}-Y^2_0\leq
  c\int_0^{t\wedge\tau(n)\wedge\tau}e^{\theta\left<M\right>_s}\left(
    Y^2_s+Z_s^2\right.\\
  +\left.f^2(s,0,0)+\kappa'\right)d\left<M\right>_s+2\int_0^{t\wedge\tau(n)\wedge\tau}e^{\frac{\theta}{2}\left<M\right>_s}Y_{s-}dN_s^\theta+\left[N^\theta\right]_{t\wedge\tau(n)\wedge\tau}.
\end{multline}
We take the supremum over $t$ on the left-hand side and afterwards the
expectation.  Recalling that
\begin{equation} \label{E31ter} N_t=\int_0^t
  e^{\frac{\theta}{2}\left<M\right>_s}Y_{s-}dN_s^\theta,
\end{equation}
this yields
\begin{multline}
  \label{E15}
  E\left(\sup_{t\leq\tau(n)\wedge\tau}\left(e^{\theta\left<M\right>_{t\wedge\tau(n)\wedge\tau}}Y^2_{t\wedge\tau(n)\wedge\tau}\right)\right)\\
  \leq E\left(Y_0^2\right)+c\,E\left(\shd\right)+2E\left(\sup_{t\geq
      0}\left|N^{\tau(n)\wedge\tau}_t\right|\right)+E\left(\left[N^\theta\right]_{\tau(n)\wedge\tau}\right),
\end{multline}
where
\begin{equation} \label{E15bis} \shd=\int_0^\tau
  e^{\theta\left<M\right>_s}\left(Y_s^2+Z_s^2+f^2(s,0,0)+\kappa'\right)d\left<M\right>_s,
\end{equation}
which has finite expectation because of \eqref{ic}.  By Remark~\ref{R1},
  \begin{align}
    E\left(\left[N^\theta\right]_{\tau(n)\wedge\tau}\right)&=
    E\left(\left<N^\theta\right>_{\tau(n)\wedge\tau}\right)\nonumber\\
    &=E\left(\int_0^{\tau(n)\wedge\tau}e^{\theta\left<M\right>_s}
      \left(Z_s^2d\left<M\right>_s+d\left<O\right>_s\right)\right).\label{31}
  \end{align}
We show now that $ N^{\tau(n)\wedge\tau}$ is a square integrable martingale.
This happens because by \eqref{E2} we have
  \begin{align*}
    E \left( \left< N \right>_{\tau(n)\wedge\tau} \right)
    &=E\left(\int_0^{\tau(n)\wedge\tau}e^{2\theta\left<M\right>_s}Y^2_s
      \left(Z_s^2d\left<M\right>_s+
        d\left<O\right>_s\right)\right)\\
    &\leq n^2E\left(\int_0^{\tau\wedge
        n}e^{2\theta\left<M\right>_s}\left(Z_s^2d\left<M\right>_s+
        d\left<O\right>_s\right)\right)\\
    &\leq n^2e^{\theta \rho(n)}E\left(\int_0^\tau
      e^{\theta\left<M\right>_s}\left(Z^2_sd\left<M\right>_s+d\left<O\right>_s\right)\right)<\infty,
  \end{align*}
taking into account Assumption~\ref{mc} at the beginning of
Section~\ref{S51}. So by Proposition 4.50 of \cite{jacod}, there is a
uniformly integrable martingale $\tilde\shm$ so that
\begin{equation*}
  \left[ N^{\tau(n) \wedge \tau }\right]=\left< N^{\tau(n) \wedge \tau}
  \right>+\tilde\shm.
\end{equation*}
Due to the Burkholder-Davis-Gundy (BDG) inequalities (see e.\,g. \cite[Theorem
IV.48]{protter}), there is a constant $c_0$ such that
\begin{equation} \label{E30} E\left(\sup_{t\geq 0}\left\vert
      N_t^{\tau(n)\wedge\tau}\right\vert\right) \leq c_0 E\left(\left[ N,
      N\right]^\frac{1}{2}_{\tau(n)\wedge\tau}\right).
\end{equation}
We denote by $\shn$ the local martingale
\begin{equation*}
  \shn_t = \int_0^t Z_s dM_s + O_t.
\end{equation*}
By Theorem 29 in Chapter II of \cite{protter} the right-hand side of
\eqref{E30} equals
\begin{multline} \label{E30bis} c_0
  E\left(\left(\int_0^{\tau(n)\wedge\tau}e^{\theta\left<M\right>_s}
      e^{\theta\left<M\right>_s}Y^2_s
      d\left[\shn\right]_s
    \right)^\frac{1}{2}\right)\\
  \leq c_0 E\left(\left(\sup_{t\leq\tau(n)\wedge\tau}
      \left(e^{\theta\left<M\right>_t}Y_t^2\right)\right)^\frac{1}{2}
    \left(\int_0^{\tau(n)\wedge\tau}e^{\theta\left<M\right>_s}
      d\left[\shn\right]_s \right)^\frac{1}{2}\right).
\end{multline}
By $2 \alpha \beta \leq\frac{\alpha^2}{c_3}+c_3\beta^2$, for any $c_3>0$, the
right-hand side of \eqref{E30bis} is bounded by
\begin{multline}
  \frac{c_0}{2c_3}E\left(\sup_{t\leq\tau(n)\wedge\tau}\left(e^{\theta\left<M\right>_t}Y_t^2
    \right)\right)+ \frac{c_0
    c_3}{2}E\left(\int_0^{\tau(n)\wedge\tau}e^{\theta\left<M\right>_s}
    d[\shn]_s
  \right)\\
  =\frac{c_0}{2c_3}E\left(\sup_{t\leq\tau(n)\wedge\tau}\left(e^{\theta\left<M\right>_t}Y_t^2
    \right)\right)+\frac{c_0
    c_3}{2}E\left(\left[N^\theta\right]_{\tau(n)\wedge\tau}\right),
\end{multline}
also using \eqref{E21} and \cite{protter}, Theorem 29, Chapter II\@.  This
gives, by \eqref{E30}, \eqref{31} and \eqref{E30bis},
\begin{multline}
  \label{E31}
  E\left(\sup_{t\geq 0}\left| N_t^{\tau(n)\wedge\tau}\right|\right)\leq\frac{c_0}{2c_3}E\left(\sup_{t\leq\tau(n)\wedge\tau}\left(e^{\theta\left<M\right>_t}Y_t^2\right)\right)\\
  +\frac{c_0
    c_3}{2}E\left(\int_0^{\tau(n)\wedge\tau}e^{\theta\left<M\right>_s}\left(Z_s^2d\left<M\right>_s+d\left<O\right>_s\right)\right).
\end{multline}
Plugging \eqref{31} and \eqref{E31} in \eqref{E15} gives
\begin{multline*}
  E\left(\sup_{t\leq\tau(n)\wedge\tau}\left(e^{\theta\left<M\right>_t}Y_t^2\right)\right)\\
  \leq
  E\left(Y_0^2\right)+E\left(\shd\right)\left(1+c+c_0c_3\right)+\frac{c_0}{c_3}E\left(\sup_{t\leq\tau(n)\wedge\tau}\left(e^{\theta\left<M\right>_t}Y_t^2\right)\right).
\end{multline*}
Choosing $c_3=2c_0$, we get
\begin{equation*}
  E\left(\sup_{t\leq\tau(n)\wedge\tau}\left(e^{\theta\left<M\right>_t}Y^2_t\right)\right)\leq 2E\left(Y_0^2\right)+2E\left(\shd\right)\left(1+c+2c_0^2\right).
\end{equation*}
By the monotone convergence theorem, letting $n\to\infty$, we get
\begin{equation*}
  E\left(\sup_{t\leq\tau}\left(e^{\theta\left<M\right>_t}Y_t^2\right)\right)\leq 
  2E\left(Y_0^2\right) + 2 E(\mathcal{D})\left(1+c+2c_0^2\right),
\end{equation*}
which shows \eqref{E1}.

We go on with the second part, i.\,e.\ the fact that $N$ defined in \eqref{E2}
is a uniformly integrable martingale. By BDG and Cauchy-Schwarz inequalities,
\begin{multline}
  E\left(\sup_{t\geq 0}\left|N_t\right|\right)\leq c_0
  E\left(\left[N,N\right]^\frac{1}{2}\right) \leq c_0
  E\left(\left(\int_0^\cdot e^{2\theta\left<M\right>_s}Y^2_{s-}
      d[\shn]_s
    \right)^\frac{1}{2}\right)\\
  \leq c_0
  E\left(\left(\sup_{t\leq\tau}e^{\theta\left<M\right>_t}Y_t^2\right)^\frac{1}{2}\left(\int_0^\tau
      e^{\theta\left<M\right>_s}
      d[\shn]_s
    \right)^\frac{1}{2}\right)\\
  \leq c_0 \left(E\left(\sup_{t\leq\tau}e^{\theta\left<M\right>_t}Y^2_t\right)
  \right)^\frac{1}{2}\left(E\left(\int_0^\tau e^{\theta\left<M\right>_s}
      d[\shn]_s
    \right)\right)^\frac{1}{2}\\
  = c_0
  \left(E\left(\sup_{t\leq\tau}e^{\theta\left<M\right>_t}Y^2_t\right)\right)
  ^\frac{1}{2}\left(E\left(\left[N^\theta\right]_\tau\right)\right)^\frac{1}{2}.
\end{multline}

By Remark~\ref{R1}
\begin{equation*}
  E\left(\left[N^\theta\right]_\tau\right)=E\left(\left<N^\theta\right>
    _\tau\right)=E\left(\int_0^\tau e^{\theta\left<M\right>_s}\left(Z^2_sd\left<M\right>_s+d\left<O\right>_s
    \right)\right)<\infty.
\end{equation*}
This shows that $N$ is a uniformly integrable martingale and finally,
Lemma~\ref{L1} is established.
\section{Proof of Theorem~\ref{thu}}
\label{ProofTheorem}
We start with some a priori bounds.  Let $\theta < \gamma$.  By
assumptions~\ref{as1} and~\ref{as2}, for any $\varepsilon\geq 0$, using
$2\alpha\beta\leq\frac{\alpha^2}{1+ \varepsilon}+ (1+\varepsilon) \beta^2$, we
can easily show that
\begin{equation}
  \label{A1}
  2\left(y-\bar y\right)\left(f(s,y,z)-f\left(s,\bar y,\bar z\right)\right)\leq -2a\left|y-\bar y\right|^2+b^2(1+\varepsilon)\left|y-\bar y\right|^2+\frac{\left|z-\bar z\right|^2}{1+\varepsilon}.
\end{equation}   
Let $\left(Y^i,Z^i,O^i\right),\,i=1,2$ be two solutions fulfilling
~\eqref{sb} of the statement. By similar arguments as \eqref{E510} and in the
lines before,
for $Y=Y^1-Y^2$, $Z=Z^1-Z^2$, $O=O^1-O^2$ we have
\begin{multline}
  \label{A2}
  e^{\theta\left<M\right>_\tau}Y^2_\tau-e^{\theta\left<M\right>_{t\wedge\tau}}Y^2_{t\wedge\tau}=\int_{t\wedge\tau}^\tau\theta e^{\theta\left<M\right>_s}Y^2_sd\left<M\right>_s+2\int_{t\wedge\tau}^\tau e^{\frac{\theta}{2}\left<M\right>_s}Y_{s-}dN^\theta_s\\
  -2\int_{t\wedge\tau}^\tau
  e^{\theta\left<M\right>_s}Y_s\left(f\left(s,Y^1_s,Z^1_s\right)-f\left(s,Y^2_s,Z^2_s\right)\right)d\left<M\right>_s+\left[N^\theta\right]_{\tau}-
  \left[N^\theta\right]_{t \wedge \tau},
\end{multline}
where $N^\theta$ was defined in \eqref{E21}.
By \eqref{A1} we get
\begin{multline}
  \label{A3}
  2\int_{t\wedge\tau}^\tau e^{\theta\left<M\right>_s}Y_s\left(f\left(s,Y^1_s,Z^1_s\right)-f\left(s,Y^2_s,Z_s^2\right)\right)d\left<M\right>_s\\
  \leq\int_{t\wedge\tau}^\tau
  e^{\theta\left<M\right>_s}\left(b^2(1+\varepsilon)-2a\right)Y^2_sd\left<M\right>_s+\int_{t\wedge\tau}^\tau
  e^{\theta\left<M\right>_s}\frac{\left|Z_s\right|^2}{1+\varepsilon}d\left<M\right>_s.
\end{multline}
$(Y^i, Z^i, O^i), i =1,2$ fulfills \eqref{ic} by \eqref{sb} and Assumption iv)
of Theorem~\ref{thu}.  Consequently $(Y,Z,O)$ also fulfills \eqref{ic}.  By
Remark~\ref{R1}, since $\theta <\gamma$, $N^\theta$ is a square integrable
martingale and
\begin{equation*}
  \left[N^\theta\right]=\left<N^\theta\right>+\shm^\theta,
\end{equation*}
where $\shm^\theta$ is a uniformly integrable martingale. So
\begin{equation}
  \label{A4}
  E\left(\left[N^\theta\right]_\tau-\left[N^\theta\right]_{t \wedge \tau} 
  \right)=E\left(\int_{t\wedge\tau}^\tau e^{\theta\left<M\right>_s}\left(Z^2_sd\left<M\right>_s+d\left<O\right>_s\right)\right).
\end{equation}
\eqref{A2}, \eqref{A3} and the fact that $Y_\tau=0$, gives
\begin{multline}
  \label{A5}
  e^{\theta\left<M\right>_{t\wedge\tau}}
  Y^2_{t\wedge\tau}+\left[N^\theta\right]_\tau
  - \left[N^\theta\right]_{t\wedge\tau}+2\int_{t\wedge\tau}^\tau e^{\frac{\theta}{2}\left<M\right>_s}Y_{s-}dN^\theta_s\\
  \leq\int_{t\wedge\tau}^\tau
  e^{\theta\left<M\right>_s}\left(b^2(1+\varepsilon)-2a-\theta\right)Y^2_sd\left<M\right>_s+\int_{t\wedge\tau}^\tau
  e^{\theta\left<M\right>_s}\frac{Z_s^2}{1+\varepsilon}d\left<M\right>_s.
\end{multline}
By Lemma~\ref{L1}, since $\theta < \gamma$,
\begin{equation*}
  \left(\int_0^{t\wedge\tau}e^{\frac{\theta}{2}\left<M\right>_s}Y_{s-}dN_s^\theta\right)_{t\geq 0}
\end{equation*}
is a uniformly integrable martingale. So its expectation is zero. By previous
considerations, \eqref{A3} and \eqref{A4}, we take the expectation in
\eqref{A5} to get
\begin{multline}
  \label{A6}
  E\left(e^{\theta\left<M\right>_{t\wedge\tau}}Y^2_{t\wedge\tau}+\int_{t\wedge\tau}^\tau e^{\theta\left<M\right>_s}\left(\frac{\varepsilon Z^2_s}{1+\varepsilon}d\left<M\right>_s+d\left<O\right>_s\right)\right)\\
  \leq E\left(\int_{t\wedge\tau}^\tau
    e^{\theta\left<M\right>_s}\left(b^2(1+\varepsilon)-2a-\theta\right)Y^2_sd\left<M\right>_s\right).
\end{multline}
Since $\theta <\gamma=b^2-2a$, we have
$b^2(1+\varepsilon)-2a-\theta>0,\,\forall\varepsilon\geq 0$. We let
$\varepsilon\to 0$ so that \eqref{A6} becomes
\begin{multline}
  \label{A7}
  E\left(e^{\theta\left<M\right>_{t\wedge\tau}}Y^2_{t\wedge\tau}+\int_{t\wedge\tau}^\tau e^{\theta\left<M\right>_s}d\left<O\right>_s\right)\\
  \leq E\left(\int_{t\wedge\tau}^\tau
    e^{\theta\left<M\right>_s}\left(b^2-2a-\theta\right)Y^2_sd\left<M\right>_s\right).
\end{multline}
Equation \eqref{A7} holds for every $\theta <\gamma$. We let
$\theta\to\gamma-$. By the monotone convergence theorem we get
\begin{equation}
  \label{A8}
  E\left(e^{\gamma\left<M\right>_{t\wedge\tau}}Y^2_{t\wedge\tau}+
    \int_{t\wedge\tau}^\tau e^{\gamma \left<M\right>_s}d\left<O\right>_s\right)\leq 0.
\end{equation}
Equation \eqref{A8} finally shows that $Y\equiv0$ and $\left<O\right>\equiv
0$. Coming back to \eqref{A6}, it easily follows that $Z\equiv 0$
$d\left<M\right>$ a.\,s.
\section{Exponential moments of the first exit time 
of Brownian motion}
\label{moments}
We consider a standard one-dimensional Brownian motion $\left\{W=W_t:t\geq
  0\right\}$, $x_0, a, b \in \R$ so that 
$x_0 \in ]a,b[$. We consider the 
  exit time
\begin{equation}
  \tau=\min\left\{s: x_0 + W_s\notin\left]a,b\right[\right\},
\end{equation}
from the interval $\left ]a,b\right[$.
 According to \cite[p.\,212]{borodin},
\begin{align}
  \label{laplace}
  E\left(e^{\gamma
    \tau}\right)&=\frac{\cosh\left(\left(b+a - 2x_0 \right)
\sqrt{-\frac{\gamma}{2}}\right)}{\cosh\left(\left(b-a\right)\sqrt{-\frac{\gamma}{2}}\right)},&\gamma\leq
  0.
\end{align}
 We define the function $f$ by
\begin{align}
  f:\C&\to\C,\nonumber\\
  z&\mapsto\sum_{n=0}^\infty\left(-1\right)^n\frac{z^n}{\left(2n\right)!}.
\end{align}
Clearly $f$ is analytical. For $x\in\R$ we can verify by Taylor series expansion that
\begin{equation}
f(x)=\left\{\begin{aligned}
&\cos\sqrt{x},&x&\geq 0,\\
&\cosh\sqrt{-x},&x&<0.
\end{aligned}\right.
\end{equation}
Let $g$ be another analytical function, defined by
\begin{align}
  g:\C\setminus\left\{\frac{\pi^2\left(2k+1\right)^2}{2\left(b-a\right)^2},k\in\N\right\}&\to\C,\nonumber\\
  \gamma&\mapsto\frac{f\left(\frac{\left(b+a - 2x_0\right)^2}{2}\gamma\right)}{f\left(\frac{\left(b-a\right)^2}{2}\gamma\right)}.
\end{align}
In particular, for $\gamma\in\R$ and $\gamma\leq 0$, this gives
\begin{equation}
\label{C6}
g\left(\gamma\right)=\frac{\cosh\left(\left(b+a - 2x_0\right)\sqrt{-\frac{\gamma}{2}}\right)}{\cosh\left(\left(b-a\right)\sqrt{-\frac{\gamma}{2}}\right)}.
\end{equation}
\begin{prop} \label{PC1} Let $\gamma \ge 0$.
  \
  \begin{enumerate}[i)]
  \item \label{p1}$E\left(e^{\gamma \tau}\right)=g\left(\gamma\right)$,
    $\gamma<\frac{\pi^2}{2\left(b-a\right)^2}$,
  \item \label{p2}$E\left(e^{\gamma \tau}\right)=\infty$,
    $\gamma\geq \frac{\pi^2}{2\left(b-a\right)^2}$.
  \end{enumerate}
\end{prop}
\begin{proof}
  ~\ref{p2} follows from~\ref{p1}, since $\gamma\mapsto E\left(e^{\gamma
    \tau}\right)$ is monotone, taking into account the Beppo-Levi convergence
  theorem. By~\eqref{laplace} and~\eqref{C6},~\ref{p1} holds for $\gamma\leq 0$, and it
  remains to show~\ref{p1} in the case
  $0<\gamma<\frac{\pi^2}{2\left(b-a\right)^2}$. For $\gamma<0$ and $n\in\N$,
  \begin{equation}
    E\left(\tau^n e^{\gamma \tau}\right)=\frac{d^n}{d\gamma^n}g\left(\gamma\right).
  \end{equation}
  Again, by the monotone convergence theorem, letting $\gamma\to 0-$, we get
  \begin{equation}
    E\left(\tau^n\right)=\left.\frac{d^n}{d\gamma^n}g\left(\gamma\right)\right|_{\gamma=0}.
  \end{equation}
  In particular all moments of $\tau$ exist. Now if
  $0<\gamma<\frac{\pi^2}{2\left(b-a\right)^2}$, then by Fubini we get
  \begin{multline}
   E\left(e^{\gamma \tau}\right)=\sum_{n=0}^\infty\frac{\gamma^n}{n!}E\left(\tau^n\right)=\sum_{n=0}^\infty\left.\frac{\gamma^n}{n!}\frac{d^n}{d\gamma^n}g\left(\gamma\right)\right|_{\gamma=0}=g\left(\gamma\right)\\
    =\frac{\cos\left(\left(b+a -2x_0\right)\sqrt{\frac{\gamma}{2}}\right)}{\cos\left(\left(b-a\right)\sqrt{\frac{\gamma}{2}}\right)},
  \end{multline}
  since $g$ is analytical on its domain. Finally~\ref{p1} follows.
\end{proof}
\bigskip

{\bfseries ACKNOWLEDGEMENTS:} The authors are grateful to the Referee for her / his
interesting comments and suggestions. \\
The research was partially supported by the ANR
Project MASTERIE 2010 BLAN-0121-01.  The first named author also benefited
partially from the support of the ``FMJH Program Gaspard Monge in optimization
and operation research'' (Project 2014-1607H).  The second named author was
supported by a ``Marietta-Blau-Stipendium'' coming from the Austrian federal
Ministry of science, research and economy (BMWF).

\printbibliography
\end{document}